\documentclass[11pt, reqno]{article}

\usepackage{epsfig,amsmath,amsthm}
\usepackage{graphicx}
\usepackage{amssymb}
\usepackage{enumerate}
\usepackage{comment}
\usepackage{color}
\usepackage{paralist}
\usepackage[left=2cm,right=2cm,top=2cm,bottom=2cm]{geometry}
\usepackage{fancyhdr}
\usepackage{hyperref}
\hypersetup{
	colorlinks=true,
	linkcolor=blue,
	citecolor=red,
	urlcolor=blue,
	pdfborder={0 0 0}
}

\usepackage{cleveref}
\crefname{theorem}{Theorem}{Theorems}
\crefname{thm}{Theorem}{Theorems}
\crefname{lemma}{Lemma}{Lemmas}
\crefname{lem}{Lemma}{Lemmas}
\crefname{remark}{Remark}{Remarks}
\crefname{rmk}{Remark}{Remarks}
\crefname{prop}{Proposition}{Propositions}
\crefname{notation}{Notation}{Notations}
\crefname{claim}{Claim}{Claims}
\crefname{defn}{Definition}{Definitions}
\crefname{definition}{Definition}{Definitions}
\crefname{cor}{Corollary}{Corollaries}
\crefname{example}{Example}{Examples}
\crefname{section}{Section}{Sections}
\crefname{figure}{Figure}{Figures}
\crefname{assumption}{Assumption}{Assumptions}

\newtheorem{theorem}{Theorem}[section]

\newtheorem{lemma}[theorem]{Lemma}
\newtheorem{cor}[theorem]{Corollary}
\newtheorem{prop}[theorem]{Proposition}

\newtheorem{example}[theorem]{Example}

\numberwithin{equation}{section}

\theoremstyle{definition}
\newtheorem{rmk}[theorem]{Remark}

\DeclareMathOperator{\cov}{cov}

\DeclareMathOperator{\GFF}{GFF}

\begin{document}
 \newcommand{\test}{test}

\newcommand{\R}{\mathbb{R}}
\newcommand{\C}{\mathbb{C}}
\newcommand{\N}{\mathbb{N}}
\newcommand{\Z}{\mathbb{Z}}
\newcommand{\I}{\mathbf{1}}
\newcommand{\E}{\mathbb{E}}
\newcommand{\HH}{\mathbb{H}}
\newcommand{\st}{\,\partial}
\newcommand{\stt}{\partial}
\newcommand{\e}{\operatorname{e}}
\newcommand{\im}{\mathrm{i}}
\newcommand{\eps}{\varepsilon}
\newcommand{\grad}{\operatorname{grad}}
\newcommand{\divg}{\operatorname{div}}
\newcommand{\Prob}[2]{\mathbb{P}_{#1} \left({#2} \right)}
\newcommand{\Prbb}[2]{\mathbf{P}_{#1} \left({#2} \right)}
\newcommand{\Prb}[1]{\mathbb{P} \left({#1} \right)}
\newcommand{\ip}[2]{\left\langle #1, #2 \right\rangle}
\newcommand{\tl}[1]{\tilde{#1}}
\newcommand{\B}{\mathbb{B}}
\newcommand{\D}{\mathbb{D}}
\newcommand{\Nz}{\mathbb{N}_0}
\newcommand{\Q}{\mathbb{Q}}
\renewcommand{\P}{\mathbb{P}}
\newcommand{\bbH}{\mathbb{H}}
\newcommand{\ga}{\gamma}
\renewcommand{\GFF}{\op{GFF}}
\newcommand{\1}{\mathbf{1}}
\renewcommand{\Re}{\mathrm{Re}}
\renewcommand{\Im}{\mathrm{Im}}
\newcommand{\scr}{\mathscr}
\def\cZ{\mathcal{Z}}
\def\cY{\mathcal{Y}}
\def\cX{\mathcal{X}}
\def\cW{\mathcal{W}}
\def\cV{\mathcal{V}}
\def\cU{\mathcal{U}}
\def\cT{\mathcal{T}}
\def\cS{\mathcal{S}}
\def\cR{\mathcal{R}}
\def\cQ{\mathcal{Q}}
\def\cP{\mathcal{P}}
\def\cO{\mathcal{O}}
\def\cN{\mathcal{N}}
\def\cM{\mathcal{M}}
\def\cL{\mathcal{L}}
\def\cK{\mathcal{K}}
\def\cJ{\mathcal{J}}
\def\cI{\mathcal{I}}
\def\cH{\mathcal{H}}
\def\cG{\mathcal{G}}
\def\cF{\mathcal{F}}
\def\cE{\mathcal{E}}
\def\cD{\mathcal{D}}
\def\cC{\mathcal{C}}
\def\cB{\mathcal{B}}
\def\cA{\mathcal{A}}
\def\cl{\mathfrak{l}}
\newcommand{\BB}{\mathbb}
\newcommand{\bS}{\mathbb{S}}
\newcommand{\ol}{\overline}
\newcommand{\ul}{\underline}
\newcommand{\op}{\operatorname}
\newcommand{\la}{\langle}
\newcommand{\ra}{\rangle}
\newcommand{\bd}{\mathbf}
\newcommand{\frk}{\mathfrak}
\newcommand{\eqD}{\overset{d}{=}}
\newcommand{\rtaD}{\overset{d}{\rightarrow}}
\newcommand{\rta}{\rightarrow}
\newcommand{\xrta}{\xrightarrow}
\newcommand{\Rta}{\Rightarrow}
\newcommand{\hookrta}{\hookrightarrow}
\newcommand{\wt}{\widetilde}
\newcommand{\wh}{\widehat} 
\newcommand{\mcl}{\mathcal}
\newcommand{\pre}{{\operatorname{pre}}}
\newcommand{\lrta}{\leftrightarrow}
\newcommand{\bdy}{\partial}
\newcommand{\tr}{\op{tr}}
\newcommand{\neu}{\op{Neu}}
\newcommand{\uo}{\op{u}}
\newcommand{\Lo}{{\op{L}}}
\newcommand{\Ro}{{\op{R}}}
\newcommand{\rad}{{\op{rad}}}
\newcommand{\cir}{{\op{circ}}}
\newcommand{\w}{{\op{wedge}}}
\newcommand{\loc}{\op{loc}}
\newcommand{\var}{\op{var}}
\newcommand{\indep}{\perp \!\!\! \perp}
\newcommand{\qbs}{\Q^{\beta,*}}

\newcommand{\noi}{\noindent}

\title{\textsc{Critical Gaussian multiplicative chaos:} \\ \textsc{a review}}
\date{}
\author{ Ellen Powell\thanks{Durham University}}
\maketitle
\vspace{-0.5cm}
\abstract{This review-style article presents an overview of recent progress in constructing and studying critical Gaussian multiplicative chaos. A proof that the critical measure in any dimension can be obtained as a limit of subcritical measures is given.}\\



\tableofcontents

\section{Introduction}

The purpose of this article is to survey and expand on recent developments in the theory of Gaussian multiplicative chaos at the so-called critical parameter. Gaussian multiplicative chaos (GMC) is the theory, originally developed by Kahane \cite{Kah85}, that aims to rigorously define measures of the form 
\begin{equation}\label{eq:gmc_formal}
\mu^\gamma(dx):=\e^{\gamma h(x)-\frac{\gamma^2}{2}\E(h(x)^2)}\, dx 
\end{equation}
when $h$ is a log-correlated Gaussian field and $dx$ is Lebesgue measure on $\R^d$. The precise definition of such a field will be given shortly, see \eqref{eqn:kx}, but the main problem is that $h$ will not be realisable as a pointwise defined function. Rather, it will only make sense as a random generalised function or Schwarz distribution, making the meaning of \eqref{eq:gmc_formal} unclear a priori. The original interest in defining such measures was to make mathematical sense of a model of Mandelbrot \cite{Man72} for energy dissipation in turbulence, but it has since found applications in many other fields, ranging from mathematical finance to Liouville quantum gravity: see \cite{RV14} for a survey.

Starting with the work of Kahane \cite{Kah85}, and now generalised and developed by many authors, it has been shown \cite{RV10, Ber17, Sha16} that one can define $\mu^\gamma$ as in \eqref{eq:gmc_formal} via a number of different regularisation procedures, as long as the parameter $\gamma$ in question is less than the critical value \begin{equation}\gamma_c=\sqrt{2d}.\end{equation} Importantly in this case, the limit $\mu^\gamma$ does not depend on the precise way that the regularisation is carried out. This justifies that $\mu^\gamma$ is the ``correct'' interpretation of \eqref{eq:gmc_formal}.  It is known that $\mu^\gamma$ is almost surely non-atomic but singular with respect to Lebesgue measure, and many further properties concerning its moments, multifractal behaviour and tail behaviour (among other things) have now been proven. See \cite{RV10,DS11,RV14,JS17, RVtail,MDW19subcrit} and the references therein for more details.

When $\gamma\ge \gamma_c$ however, a different picture emerges. The regularisation scheme that works perfectly when $\gamma<\gamma_c$ now fails, in the sense that it yields a trivial $(\equiv 0)$ measure in the limit, \cite{RV10}. This begs the question of whether something meaningful can still be defined in this regime. Since the subcritical measures become more and more localised as $\gamma$ increases (roughly speaking, they live on a set of Hausdorff dimension converging to $0$) the natural guess is that any definition of a supercritical measure will be atomic, and this indeed turns out to be the case: see \cite{RV14} for a summary. The value $\gamma=\gamma_c$ is particularly intriguing, as it represents the transition between these two different types of measure. For a discussion of the relevance of this phase transition from a  physical perspective, the reader is referred to the introductions of \cite{DRSV14one,DRSV14two}.

As will be discussed in this article, there is a rich mathematical theory that emerges at the critical value $\gamma=\gamma_c$, and it is possible to define a canonical ``critical measure'', denoted $\mu'$, which \emph{is} still non-atomic but only barely so \cite{DRSV14one,DRSV14two,JS17,HRV18,Pow18chaos,JSW19}. The article will be structured as follows.
\begin{itemize}
	\item First, the various constructions of critical chaos will be stated in their most general form, and some ideas will be given for the proofs. Several important properties will also be discussed.
	\item Next, the idea of constructing critical chaos as a limit, or derivative, from the subcritical regime will be explored. As part of the section a proof will be given for this construction in the general setting, which seems to be new.
	\item Finally, some applications of critical chaos will be surveyed, which are surprisingly far-reaching and cover topics currently undergoing rapid and active investigation.
\end{itemize}

The main aim of this article is to showcase the key ideas underlying the construction(s) of critical chaos, and to hopefully not get too bogged down in technicalities (of which there are many). It is intended that analogies with the theory of branching random walks and branching Brownian motion should never be too far from the forefront of the exposition.

\subsection*{Acknowledgements}

The article was written following the Inhomogeneous Random Systems conference in Paris, 2020. I would like to thank Fran\c{c}ois Dunlop and Ellen Saada for the organisation of this illuminating meeting, and Nathana\"{e}l Berestycki for the invitation to attend and give a talk.

I would also like to thank Marek Biskup, Yan Fyodorov, Antoine Jego, Gaultier Lambert and Fr\'{e}d\'{e}ric Ouimet for providing many helpful comments concerning the ``applications'' section of the paper.

\section{Construction and properties}\label{sec:const}

\subsection{Log-correlated Gaussian fields} \label{sec:general_fields}

\begin{def}\label{def:general_fields}
In this article, a log-correlated Gaussian field refers to a centred Gaussian field $X$ defined on a domain $D\subseteq \R^d$, whose covariance kernel $K_X$ satisfies
\begin{equation}\label{eqn:kx}
K_X(x,y)=-\log(|x-y|)+g(x,y),
\end{equation}
with $g\in H_{\loc}^{d+\eps}(D\times D)$ for some $\eps>0$.
\end{def}

This is the most general class of field for which results concerning critical Gaussian multiplicative chaos exist, \cite{JSW19}. Such a field $X$ cannot be defined as a function assigning values to the points of $D$, but does make sense as a random generalised function, or distribution in the sense of Schwarz. This random distribution will in fact have some additional regularity properties. For example, it will almost surely be an element of $H^{-\eps}(D)$ for any $\eps>0$, \cite[Proposition 2.3]{JSW18}.  

There are some special types of field that are particularly nice to study in the context of Gaussian multiplicative chaos, due to  explicit and useful decorrelation properties. As such they have played an important role in the development of the theory, as well as often being mathematically and physically relevant for separate reasons. For example: 

\begin{itemize}
	\item The \emph{planar Gaussian free field} with zero boundary conditions on a simply connected domain $D\subset \C=\R^2$ is the centred Gaussian field  with covariance kernel \begin{equation} \label{eqn:GFF_kernel} K(x,y)=G_D(x,y), \end{equation}
	where $G_D$ is the Green's function for Brownian motion killed when leaving $D$. This field is particularly important due to two special properties: conformal invariance and a spatial Markov property. These properties actually characterise the field \cite{BPR18}, making it in some sense a ``universal object" describing the fluctuations of random planar height functions. See \cite{Ber16,SheffieldGFF,WPgff} for a more detailed introduction. As will be mentioned later in this article, the Markov property gives rise to some martingales that make constructing chaos for the GFF somewhat more straightforward than in the general case. 
	
	It should also be mentioned that chaos measures for the GFF are important objects from a physical perspective. Very roughly speaking, they are supposed to represent the volume form of a ``uniformly chosen random surface" weighted by the partition function of some statistical physics model. They are often referred to as ``Liouville quantum gravity measures'' in the probability literature.
	
	\item One can alternatively consider Gaussian free fields with non-zero boundary conditions on $D$; a natural example being the Gaussian free field with ``Neumann'' or ``free'' boundary conditions (again see \cite{Ber16,SheffieldGFF}). This has covariance kernel given by the Green's function in $D$ with \emph{Neumann} boundary conditions, with the caveat that a lack of uniqueness for this kernel only defines the field up to an additive constant. One can fix the additive constant in any number of ways to get back to the setting of \eqref{eqn:kx}.
	\item For example, taking $D=\D$ and $K(x,y)=-\log(|x-y||1-x\bar{y}|)$ yields the GFF with ``vanishing mean on the unit circle". When this field is restricted to the unit circle $\partial \D$ it gives rise to a centred Gaussian field with covariance
	\begin{equation}\label{eqn:gff_cov_circle}
	K(\e^{i\theta},\e^{i\theta'})=-2\log(|e^{i\theta}-e^{-i\theta}|) \text{ for } \theta,\theta'\in [0,2\pi],
	\end{equation}
	sometimes referred to as the ``GFF on the unit circle''. Note that one can reparametrise by $\theta$ to define a field on $[0,2\pi]\subset \R$, but should divide by $\sqrt{2}$ to get a 1d log-correlated field as in \eqref{eqn:kx}. These fields play an important role in extreme value theory; for example in connection to random matrices. See \cite{Arg16} for a survey on this, and also the discussion in \cref{sec:extrema} below.
	\item Finally, there is a class of log-correlated Gaussian fields on $D=\R^d$ whose covariance kernel has the special integral form,
	\begin{equation}\label{eqn:ssi}K(x,y)=\int_1^\infty \frac{k(u(x-y))}{u} \, du\end{equation}
	for some $k:\R^d\to \R$ such that $(x,y)\mapsto k(x-y)$ is a covariance on $\R^d$. For the rest of this article, the definition of $\star$-scale invariant will also mean that $k$ is rotationally symmetric, continuously differentiable, supported in $B(0,1)$ and with $k(0)=1$. Under these conditions, $K$ defines the covariance kernel of a log-correlated field as in \eqref{eqn:kx}, see \cite{DRSV14one,DRSV14two, JSW19}.
	These play an important role for a couple of reasons: firstly, there is a particularly natural way to approximate these fields; and secondly, their associated chaos measures satisfy a certain scaling relation. These will both be discussed in the next section.

\end{itemize}

\subsection{Approximations}

Recall from the introduction that the Gaussian multiplicative chaos associated to $X$ as in \eqref{eqn:kx} is formally given by the measure
\begin{equation}\label{eqn:gmc_formal}
\e^{\gamma X(x) - \frac{\gamma^2}{2} \E(X(x)^2)} \, dx 
\end{equation} 
on $D$, for some $\gamma\ge 0$. Since $X$ is not a pointwise defined function, this does not make sense a priori, and so one needs to define it via a regularisation procedure.

One natural way to approximate, or regularise, a rough (distribution-valued) Gaussian field as in \cref{sec:general_fields}, is to convolve it with a smooth approximation to the identity. More precisely, if $X$ is defined on a subset of $\R^d$ and $\psi\ge 0$ is a smooth function with compact support and total integral one, one considers 
\begin{equation}
\label{eqn:xeps}
X_\eps := \psi_\eps * X
\end{equation}
for each $\eps>0$, where $\psi_\eps(\cdot)=\eps^{-d} \psi(\eps^{-1 \cdot})$. 
Then for any $x,y\in D$ such that $B(x,\eps),B(y,\eps)\subset D$, the covariance between $X_\eps(x)$ and $X_\eps(y)$ is given by the (double) convolution of $K_X$ with $\psi_\eps$. The field $X_\eps$ therefore has covariance close to that of $K$ for small $\eps$, but is a bona fide continuous random field. Such an approximation will be referred to forthwith as a \emph{convolution approximation} to the field.
 
For particular types of field there are some other natural approximations. The examples below are especially convenient to work with, since they give rise to approximate chaos measures with a simple martingale structure, which is of course helpful when trying to prove convergence results. Moreover, they have direct counterparts in the setting of branching random walks, whose behaviour has been extensively studied and is by now very well understood. This point will be expanded on throughout the present article.

\begin{itemize}
	\item 	Suppose that $X$ is a $\star$-scale invariant field with associated function $k$. Then 
	\begin{equation}
	\label{eqn:starcut}
	K_t(x,y)=\int_1^{\e^t} \frac{k(u(x-y))}{u} \, du 
	\end{equation} is the covariance kernel of a field $X_t$, approximating $X$ as $t\to \infty$. Moreover, the fields $(X_t;\, t\ge 0)$ can be coupled, \cite{ARV13}, so that $(X_{t}(x);\, t\ge 0)$ has independent increments for any $x$, and such that $(X_{t}(x)-X_{t_0}(x); {t\ge t_0})$ and $(X_{t}(y)-X_{t_0}(y);\, {t\ge t_0})$ are independent for any $x,y$ with $|x-y|\ge \e^{-t_0}$. This readily implies, for example, that for any $x\in \R^d$, $(X_t(x);\, t\ge 0)$ has the law of a standard linear Brownian motion started from $0$.
	
	For ease of reference, let us refer to such a coupled sequence of approximations $(X_t; t\ge 0)$ as the ``$\star$-scale cut-off approximations'' to $X$.
	
	\item	The form of the kernel in the $\star$-scale invariant case means that chaos measures for such fields satisfy a special scaling relation, \cite{DRSV14one}, called the ``$\star$-equation''. More precisely, if $\mu^\gamma$ is a subcritical ($\gamma<\sqrt{2d}$) chaos measure for a $\star$-scale invariant field, then for any $t\ge 0$
	\begin{align}
	\label{eqn:star}
	(\mu^\gamma(A); \,A\in \cB(\R^d)) & \overset{(\mathrm{law})}{=} (\int_{\R^d} \e^{\gamma X_t(x)-\frac{\gamma^2}{2}\E(X_t(x)^2)}\e^{-dt}\mu^{\gamma,t}(dx) ;\, A\in \cB(\R^d)) \text{ where } \\ (\mu^{\gamma,t}(A);\, A\in \cB(\R^d)) & \overset{(\mathrm{law})}{=} (\mu^\gamma(\e^t A); \,A\in \cB(\R^d)) \text{ and } \mu^{\gamma,t} \indep (X_t(x); x\in \R^d)
	\end{align} 
	(and the law of $(X_t(x);\, x\in \R^d)$ is as described in the previous bullet point.
	
	This says that the measure looks the same after rescaling space, up to a smooth independent Gaussian change of measure. Solutions of this equation are analogous to fixed points of the smoothing transform for the branching random walk (see \cite{BK05} for a review).
	
	\item When $X$ is a planar Gaussian free field with zero boundary conditions, one can define the \emph{circle average} around any point in the domain, by taking a sequence of smooth test functions approximating uniform measure on the circle and then taking a limit of $X$ tested against these functions. It can easily be shown that these limits exist for every circle contained in the domain, and moreover, that there exists a version of the process that is almost surely jointly continuous in the centre and radius of the circle, \cite{HMP10}. Write $h_\eps(z)$ for the average on the circle $\partial B_\eps(z)$. Then the Markovian property of the GFF means that for any $z\in D$, the process $t\mapsto X_{\e^{-t}}(z)$ is continuous and centred, with stationary and independent increments.  It is therefore (some multiple of) a Brownian motion. In fact, the Markov property further gives that if $x$ and $y$ are two distinct points, then $(X_{\e^{-t}}(x)-X_{\e^{-t_0}}(x);\,{t\ge t_0})$ and $(X_{\e^{-t}}(y)-X_{\e^{-t_0}}(y); \, {t\ge t_0})$ are  independent if $t_0\ge -\log|(|x-y)/2|$.
	\item There is another special way to approximate the planar Gaussian free field, that is even nicer for the purposes of constructing chaos measures. This is because it very closely links the Gaussian free field with a branching random walk, and importantly, provides approximate GMC measures with a martingale property. This approximation relies on a beautiful coupling between the Gaussian free field and a conformal loop ensemble with parameter 4, CLE$_4$. The coupling is due to Miller and Sheffield \cite{MSCLE}; see also \cite{BTLS} for a proof.
	
	For the purposes of this article, let us just mention very briefly that CLE$_4$ is a random collection of disjoint simple loops defined in the unit disc, whose law is conformally invariant, and such that the union of the interiors of the loops has full Lebesgue measure. This means that CLE$_4$ can be unambiguously defined in any simply connected domain (by mapping to the unit disc) and in particular, the construction can be iterated inside each loop of the CLE$_4$. This allows one to define an (infinitely) nested version of the conformal loop ensemble. One can then associate a branching random walk to the branching sequences of nested loops, where the walk steps are Bernoulli $\pm 1$. If one stops this branching random walk after $n$ steps, and assigns the relevant value to the interior of each $n$th level nested CLE loop, then this provides a function $X_n(x)$ that is defined at all but a Lebesgue-null set of points $x$ in the disc (it is constant inside each loop). Somewhat incredibly, it holds that \begin{equation}
	\label{eqn:cle_approx}
	 X_n(x)\to X \text{ as } n\to \infty,
	\end{equation} in probability (as generalised functions), where $X$ is a multiple of the zero boundary GFF on the disc. 

  \end{itemize}



\subsection{The phase transition}\label{sec:phase_trans}

Recall from the introduction that if $(X_\eps)_{\eps\ge 0}$ are a sequence of suitable approximations to the field, the method for constructing subcritical GMC measures $(\gamma<\sqrt{2d})$ is to take a limit of approximate measures 
\begin{equation}
\label{eqn:mueps}
\mu_\eps^\gamma(dz):= \exp(\gamma X_\eps(z)-\frac{\gamma^2}{2}\E(X_\eps(z)^2)) \, dz
\end{equation} 
as $\eps\to 0$. The limiting measure exists in probability, and is almost surely non trivial \cite{Kah85,RV10,Ber17,Sha16}. Attempting to do the same when $\gamma\ge \sqrt{2d}$, one encounters a phase transition. That is:

\begin{lemma}\label{lem:conv_to_0}
Suppose that $(X_\eps)_{\eps \ge 0}$ are a sequence of convolution approximations to a log-correlated Gaussian field $X$ as in \cref{sec:general_fields}. Define $\mu_\eps^\gamma$ as in \eqref{eqn:mueps}. Then for any $A\subset D$ compact, $\mu^\gamma_\eps(A)\to 0$ in probability as $\eps\to 0$. 
\end{lemma}

This is a well understood phenomenon in the branching random walk literature, and it is instructive to see how the proof works in this case. In fact, one can use this to prove \cref{lem:conv_to_0}, as was done in \cite[Appendix]{DRSV14one}: see below. \\

Take a branching random walk starting with one particle at a random position $W$, where at each stage $n$ for $n\ge 1$, particles die and give birth to exactly $2d$ children particles. Suppose that the position of each of these children is displaced by an independent copy of $W$ from the position of their parent. For concreteness, take $W\sim N(0,\ln(2))$ (although the argument works in a much more general setting) and write $\mathbb{P}$ for the law of this process. 

It is well known and easy to check that if $(X_n(1),\cdots, X_{n}(2^{dn}))$ are the positions of the particles at time $n$, then 
\begin{equation}
\label{eqn:mnb_def} M^\gamma_n := \prod_{i=1}^{2^{dn}} \e^{\gamma X_n(i) - (\gamma^2/2+d)n\ln 2 } \end{equation}
is a positive martingale with respect to the natural filtration $(\mathcal{F}_n ; n\ge 0)$ of the branching random walk. It therefore has an almost surely finite limit as $n\to \infty$.

A useful criterion to determine whether the limit is trivial, comes from  the following basic fact, \cite{Durbook}. If $\mathbb{Q}$ is the probability measure such that $(d\mathbb{Q}/d\mathbb{P})|_{\mathcal{F}_n}=M_n^\gamma$ for each $n$ then $M_n^\gamma \to 0$ under $\mathbb{P}$ if and only if $\limsup_{n\to \infty} M_n^\gamma = \infty$ under $\mathbb{Q}$ (and this is further equivalent to $\mathbb{Q}$ being singular with respect to $\mathbb{P}$). One can determine rather easily if this is the case, since there is a very nice description of the behaviour of the process under $\mathbb{Q}$ known as the ``spine decomposition". This is an extensively used technique in the branching process literature: see \cite{HRspine} for a general formulation.

To describe the decomposition in this set-up, first write $\mathbb{P}^*$ for the law of the branching random walk, plus a distinguished ``spine'' chosen uniformly at random. That is, take the spine particle at generation $0$ to be the initial particle, and then iteratively choose the spine particle at generation $n$ by picking one of the children of the spine particle at generation $(n-1)$ uniformly at random. The motion $(X_n^*;n\ge 0)$ of the spine particle under $\mathbb{P}^*$ is just a random walk with $N(0,\ln(2))$ increments, and so $\xi_n:=\exp(\gamma X_n^*-(\gamma^2/2)n\ln 2)$ is a unit-mean martingale. Define $\mathbb{Q}^*$ by setting $(d\mathbb{Q}^*/d\mathbb{P}^*)|_{\cF_n^*}=\xi_n$ for each $n$, where $(\cF_n^*; n\ge 0)$ is the filtration generated by the branching random walk \emph{and the spine} up to generation $n$. It then follows easily from Girsanov's theorem that under $\mathbb{Q}^*$: the spine particle evolves as a random walk with $N(0,\ln 2)$ increments \emph{plus} a drift of $\beta \ln 2$; and at each stage the spine gives birth to exactly $(2d-1)$ non-spine children, who each start (from their respective positions) an independent $\mathbb{P}$-branching random walk. The point of all this is that $(d\mathbb{P}^*/d\mathbb{Q}^*)|_{\cF_n}=M_n^\gamma$ for each $n$, so that $\Q^*$ and $\Q$ give the same marginal law to the branching random walk. In particular,  if $\limsup_{n\to \infty}M_n^\gamma$ a.s.\ under $\mathbb{Q}^*$, then the same holds under $\mathbb{Q}$.

It turns out that the knowledge the spine particle under $\mathbb{Q}^*$ makes this condition easy to check. Indeed, it is certainly true that $M_n^\gamma$ is bigger than $\exp(\gamma X_n^* - (\gamma^2/2+d)n\ln 2)$, where under $\mathbb{Q}^*$,  $(\gamma X_n - (\gamma^2/2+d)n\ln 2;\, {n\ge 0})$ is a random walk with $N(0,\gamma^2\ln 2)$ increments plus a drift of $(\gamma^2/2-d)n \ln 2$. This implies that if $\gamma^2\ge 2d$, $\limsup_{n\to \infty} M_n^\gamma= \infty$ almost surely under $\mathbb{Q}^*$. As explained above, this implies the same under $\mathbb{Q}$, and hence that $M_n^\gamma \to 0$ under $\mathbb{P}$. \\

To prove \cref{lem:conv_to_0} we can essentially use the above, together with an extremely useful comparison inequality due to Kahane, \cite{Kah85}. 

\begin{theorem}[Kahane's convexity inequality]
	\label{thm:kahane} Suppose that $Z_1$ and $Z_2$ are two  almost surely continuous centred Gaussian fields defined on $D\subset \R^d$, with $\mathbb{E}[Z_i(x)Z_i(y)]=K_i(x,y)$ for $i=1,2$. Suppose further that 
 $K_1(x,y)\le K_2(x,y)$ for all $x,y$. Then if $F:(0,\infty)\to \R$ is any convex function that grows at most polynomially fast at $0$ and $\infty$, it holds that 
	\begin{equation}
	\label{eqn:kah}
	 \mathbb{E}(F(\int_D \exp( Z_1(x)-\frac{1}{2}\E(Z_1(x)^2)) \, dx ))\le \mathbb{E}(F(\int_D \exp( Z_2(x)-\frac{1}{2}\E(Z_2(x)^2)) \, dx )). \end{equation}
\end{theorem}

\begin{proof}[Proof of \cref{lem:conv_to_0}]
In order to make use of the preceding discussion, one needs to find a way to relate $M_n^\gamma$ to a ``chaos measure'' of some sort. In fact, there is a very natural way to do this. Without loss of generality, assume that $A$ is the unit cube $[0,1]^d\subset \R^d$. 
If $X_0$ is the initial position of the branching random walk, define $Y_0$ to be the constant function equal to $X_0$ on $[0,1]^d$. At the next stage, divide the unit cube into $2d$ sub-cubes of side length $1/2$ and set $Y_1$ to be constant in each of these sub-cubes: equal to $X_1(i)$ in sub-cube $i$, where the sub-cubes are ordered in some way and $X_1(1),\cdots, X_1({2d})$ are positions of the particles in the branching random walk at generation one.  Iterating this procedure defines a centred Gaussian field $Y_n$ on $[0,1]^d$ for every $n$, for which \begin{equation}M_n^\gamma=\int_{[0,1]^d} \exp(\gamma Y_n(x) - \frac{\gamma^2}{2}\E(Y_n(x)^2)) \, dx.\end{equation} So from the discussion above \cref{thm:kahane}, it holds that $M_n^\gamma\to 0$ whenever $\gamma\ge \sqrt{2d}$.
On the other hand, if $x,y\in [0,1]^d$ are at distance greater than $\sqrt{d} 2^{-n}$, then $Y_n(x)$ and $Y_n(y)$ correspond to positions of particles in the random walk that ``branched from each other'' before stage $n$, and so $\E(Y_n(x)Y_n(y))\le -\log(|x-y|)+C_d$ for some  $C_d<\infty$ depending only on $d$. 

Finally, if $(X_\eps;\, {\eps\ge 0})$ is as in the statement of the lemma, then for each $\eps>0$ we can find $n,C$ so that Kahane's inequality applies to $Z_1=\gamma Y_n$ and $Z_2=\gamma X_\eps+C\cN$ with $\cN$ a standard normal, independent of $X_\eps$.\footnote{There is a slight technicality here, since the field $Z_1$ is not actually continuous, and so \cref{thm:kahane} as stated cannot be applied. However, one can check (see \cite[Proof of Lemma 1]{Kah85}) that the proof works for a field such as $Z_1$, for which the chaos measure clearly exists and has a density with well-behaved tails.} This yields that for any $\gamma\ge \sqrt{2d}$ and $F$ bounded, concave and increasing
\begin{equation}\limsup_{\eps\to 0} \E(F(\e^{\gamma C\cN-\gamma^2C^2/2} \mu^\gamma_\eps([0,1]^d)))\le \limsup_{n\to \infty} \E(F(M_n^\gamma))=\E(F(0)).\end{equation}
As a consequence (see for example \cite[Appendix]{DRSV14one}) it holds that $\mu_\eps^\gamma([0,1]^d)\to 0$ in probability as $\eps\to \infty$. 
\end{proof}\medskip

The borderline case $\gamma=\sqrt{2d}$ is referred to as the \emph{critical} regime of GMC, and the case $\gamma>\sqrt{2d}$ the \emph{subcritical regime} accordingly. 

\begin{rmk}\label{rmk:conv_to_0_star}
	Consider the $\star$-scale cut-off approximations $(X_t;\, t\ge 0)$ to a $\star$-scale invariant field as described in the paragraph surrounding \eqref{eqn:starcut}. The same argument as above gives that if $\mu_t^{\gamma}(dx):=\exp(\gamma X_t(x)-(\gamma^2/2)\E(X_t(x)^2)) \, dx$ are the associated approximate chaos measures for $\gamma\ge \gamma_c$, then $\mu_t^\gamma(A)\to 0$ in probability as $t\to \infty$ for any $A\subset \R^d$ compact. In fact, in this case it is easy to see that $\mu_t^{\gamma}(A)$ is martingale for each $A$. This means that the convergence actually holds almost surely. 
\end{rmk}

\subsubsection{Spine decomposition/Rooted measures}\label{sec:spine}
At this point it is useful to discuss an analogous ``spine decomposition'' that one has for Gaussian multiplicative chaos measures. For clarity, and because this case will be used in detail later on, let us assume that $(X_t)_{t\ge 0}$ is a $\star$-scale cut-off approximation to a $\star$-scale invariant field, as in \eqref{eqn:starcut}.

Recall that $(X_t(x); \,t\ge 0)$ is a standard Brownian motion for each $x$, so in particular, $\E(X_t(x)^2)=t$ for all $t,x$. If \begin{equation}\label{eqn:mutg} \mu^\gamma_t(dx):= \exp(\gamma X_t(x)-(\gamma^2/2)t) \, dx\end{equation} then for any bounded subdomain $A\subset \R^d$ one can define a measure $\mathbb{Q}$ such that \begin{equation}\frac{d\Q}{d\P}|_{\cF_t}=|A|^{-1}\mu_t^\gamma(A),\end{equation} where $\cF$ is the filtration generated by $(X_s(x);\, x\in A, s\le t)$. One can also append a uniformly chosen point $x^*$ in $A$ to $\mathbb{P}$ (independently of $\cF$), thereby defining a new measure $\P^*$. 

Since $\zeta_t=\exp(\gamma X_t(x^*)-(\gamma^2/2)t)$ will be a martingale  under $\P^*$ (the classical Brownian motion exponential martingale), one can define a measure $\Q^*$ by \begin{equation}\frac{d\Q^*}{d\P^*}|_{\cF^*_t}=\zeta_t\end{equation} for each $t$. Here as in the branching random walk case $\cF^*_t$ is the filtration generated by $x^*$ together with $(X_s(x); x\in A, s\le t)$. 

Then just as before, one has that $(d\Q^*/d\P^*)|_{\cF_t}=|A|^{-1}\mu_t^\gamma(A)$ so that the $\Q$- and $\Q^*$-laws of $(X_t(x); x\in \R^d, t\ge 0)$ are identical. On the other hand, it is easy to check that:
\begin{itemize}
	\item  the $\Q^*$-law of $x^*$ given $\cF_t$ is proportional to $\mu_t(dx)\I_{A}$;
	\item for any $t$, the $\Q^*$-law of $(X_t(x); x\in A, t\ge 0)$ given $x^*$ is, by Girsanov's theorem, that of a Gaussian process with the same covariance structure as $(X_t(x); x\in A)$, but with mean given by $\gamma\E(X_t(x)X_t(x^*))$ at the point $x$;
	\item the $\Q^*$ law of $(X_t(x^*); t\ge 0)$ is that of a standard Brownian motion plus a drift of $\gamma t$.
\end{itemize}

\subsection{Renormalisation}\label{sec:renorm}

From now on, the focus will be on the case \begin{equation}
\label{eqn:gc}
\gamma=\gamma_c:=\sqrt{2d}.\end{equation} There is a separate and equally interesting story when $\gamma>\sqrt{2d}$, but it is not within the scope of the present article. See \cite[Section 6]{RV14} for an overview.

When $\gamma=\gamma_c$ and $X$ is a field as in \eqref{eqn:kx}, the previous section shows that convolution approximations $\mu_\eps^{\gamma_c}$ as defined in \eqref{eqn:mueps} converge to $0$ as $\eps\to 0$. On the other hand, for any $\gamma<\gamma_c$, $\mu_\eps^\gamma$ converges to a non-trivial limiting measure. Thus one can hope that by giving the sequence $\mu_\eps^{\gamma_c}$ an appropriate ``push'' in the right direction, something interesting may still be obtained. 

The next question is how, precisely, to do this. A simple solution would be to take a limit of measures $c_\eps \mu^{\gamma_c}_\eps$ with some $c_\eps$ deterministic and converging to $\infty$ as $\eps\to \infty$. But of course, $c_\eps$ must be chosen carefully, and it is not so clear a priori what it should be. 

Another approach arises from the following example.

\begin{example}\label{ex:BMDM} 
	
	Recall the setting of \cref{sec:spine}, and note that for each fixed $x\in \R^d$, the martingale $\exp(\gamma X_t(x)-(\gamma^2/2)t)$ is the classical ``exponential martingale'' of $\gamma$ times a Brownian motion. It is standard and easy to check that this martingale is \emph{not} uniformly integrable, and converges to $0$ almost surely as $t\to \infty$.
	
	Therefore, the only way that $\mu_t^\gamma$ can have a non-trivial limit,  is that there are enough \emph{exceptional} points where this martingale is atypically large. In fact, from the discussion in \cref{sec:spine} it can be deduced that if $\mu_t^\gamma$ does converge in $L^1(\mathbb{P})$, then a point $y$ sampled with probability proportional to the limiting measure will have 
	\begin{equation}\label{eqn:thick}
	\lim_{t\to \infty} \frac{X_t(y)}{t} = \gamma
	\end{equation}
	almost surely. This is in contrast to a fixed deterministic point, for which the above limit would correspond to $\lim_{t\to \infty}B_t/t$ and almost surely be 0.
	
	When $\gamma<\gamma_c$ the convergence in $L^1(\P)$ is known, and so there are indeed enough such points to support the measure. When $\gamma=\gamma_c$ however, it can be shown that the Hausdorff dimension of points satisfying \eqref{eqn:thick} is zero, and for any bounded $A\subset \R^d$ 
	\begin{equation}\label{eqn:extrema_basic}
	\limsup_{t\to \infty} \sup_{x\in A} (X_t(x)-\gamma_c t) = -\infty
	\end{equation}
a.s. (See \cref{sec:extrema} for a more refined statement). Note that for the branching random walk case, \eqref{eqn:extrema_basic} follows from the fact that $M_n^{\gamma_c}\to 0$ as $n\to \infty$. The result for $\star$-scale invariant fields can be deduced by a comparison argument, \cite{DRSV14one}, and for general fields by comparison with a modified branching Brownian motion, see e.g. \cite{Aco14}.

So essentially, the reason that $\mu_t^{\gamma_c}$ converges to $0$, is that the martingale $\exp(\gamma_c X_t(x)-(\gamma_c^2/2)t)$ is just too small everywhere. Nevertheless, there is an alternative martingale that is natural to consider. This arises by taking the \emph{derivative}  of  $\exp(\gamma_c X_t(x)-(\gamma_c^2/2)t)$ with respect to the parameter $\gamma$. That is, considering \begin{equation} d_t(x):=(-X_t(x)+\gamma_c t)\exp(\gamma_c X_t(x)-(\gamma_c^2/2) t).\end{equation}
Note that this is indeed a martingale, by its definition as the derivative of a martingale. 

 In fact, if one modifies this slightly and considers 
 \begin{align}\label{def:dtbatb}
 d_t^\beta(x) &: = (-X_t(x)+\gamma_c t+\beta)\exp(\gamma_c X_t(x)-(\gamma_c^2/2)t)\I_{A_t^\beta(x)} \\ A_t^\beta(x) & :=\{(-X_s(x)+\gamma_c s)\ge -\beta \;\;\forall s\in [0,t]\}
 \end{align} instead (in order to work with a positive density) then this is again a martingale. Moreover, for fixed $x$ if one changes measure using this martingale, then $(-X_t(x)+\gamma_c t+\beta)$ under the new measure will have the law of a $3d$-Bessel process started from $\beta$; see \cite{RYbook} for justification of these facts, and much more information on general Bessel processes.
 
 This means, very roughly, that a limit of
 \begin{equation}
 D_t^\beta(dx):= d_t^\beta(x)dx
 \end{equation}
as $t\to \infty$ (if it exists) \emph{could} be supported on points $x$ where $-X_t(x)-\gamma_c t $ goes to $-\infty$. Therefore one can be more hopeful (and it turns out rightly so) that such a limit exists. In fact, since $D_t^\beta(dx)$ and $\mu_t'(dx):= d_t(x)\, dx$ are extremely close for large $\beta$ (by \eqref{eqn:extrema_basic} and \cref{lem:conv_to_0}) then this should actually imply that $\mu_t'$ itself converges.
\end{example}

The analogue of this in the setting of convolution approximations is to consider

\begin{align}\label{eqn:der_norm_eps}
\mu_\eps'(dx) &:= -\frac{d}{d\gamma} \left( \exp(\gamma X_\eps(x)-\frac{\gamma^2}{2}\E(X_\eps(x)^2)) \right) \big|_{\gamma=\gamma_c} \, dx \nonumber \\  
& = \left(- X_\eps(x)+\gamma_c \E(X_\eps(x)^2)\right) \, \exp(\gamma_c X_\eps(x)-\frac{\gamma_c^2}{2}\E(X_\eps(x)^2)) \, dx
\end{align}
and ask if 

\begin{equation}\label{eqn:der_norm}
\lim_{\eps\to 0} \mu_\eps'(dx) 
\end{equation}
exists. If it does, then this in turn provides a guess for a ``deterministic renormalisation'' that may work. Indeed, the discussion in \cref{ex:BMDM} suggests that a limit of the form \eqref{eqn:der_norm} should be ``supported'' on points where $(-X_\eps(x)+\gamma_c X_\eps(x))$ behaves like a 3d-Bessel process at time $\log(1/\eps)$. Since a 3d-Bessel process at time $t$ is typically of size $\sqrt{t}$, a reasonable guess is that
\begin{equation}\label{eqn:sh_norm}
\lim_{\eps\to 0} \sqrt{\log(1/\eps)} \mu_\eps^{\gamma_c}.
\end{equation}
may exist and be non-trivial. This normalisation is known as the ``Seneta--Heyde'' normalisation, and has a well-established counterpart in the branching random walk literature \cite{AiSh}. In the $\star$-scale invariant setting of \cref{ex:BMDM}, one instead tries to take a limit of $ \sqrt{t}\mu^{\gamma_c}_t$ as $t\to \infty$.

Another, related but somewhat more na\"{i}ve, line of reasoning is the following. For $\gamma<\gamma_c$ it is known that the limit measure $\mu^\gamma=\lim_{\eps\to 0}\mu^\gamma_\eps$ exists and is non-trivial, while on the other hand, $\lim_{\eps\to 0} \mu^{\gamma_c}_\eps =0$. Since everything seems to depend pretty nicely on $\gamma$, it is not too hard to believe that the measures $\mu^\gamma$ are in some sense regular with respect to $\gamma$, and converge to $0$ as $\gamma\uparrow \gamma_c$. If this were true, then even though the limit of $\mu^\gamma$ as $\gamma\uparrow \gamma_c$ is trivial, the \emph{rate} at which it approaches zero may not be. This leads one to ask if  

\begin{equation} \label{eqn:derc_norm} \lim_{\gamma\uparrow \gamma_c} \frac{\mu^\gamma}{\gamma_c-\gamma} \end{equation}
 exists in some appropriate sense.  This turns out to be rather tricky to address, but it is natural to expect that, if they exist, \eqref{eqn:derc_norm} and \eqref{eqn:der_norm} should be closely related. Indeed, they correspond to one another up to exchanging the order of limit $\gamma\to \gamma_c$ and $\eps\to 0$. This point will be explored in some detail in \cref{sec:der}.

It turns out, as will be discussed in the next sections, that all of these approaches are essentially equivalent, and correct.

\subsection{Construction of critical Gaussian multiplicative chaos}

\cref{thm:const_gen} of this section is the most general statement to date concerning the convergence of \eqref{eqn:der_norm_eps} and \eqref{eqn:sh_norm}. This final version is due to \cite{JSW19}, and builds on a series of works \cite{DRSV14one,DRSV14two,JS17,HRV18,Pow18chaos} that will be summarised just below.

As in \cite[Definition 5.1]{JSW19}, if $(\mu,(\mu_n; n>0))$ are random measures on some compact set $K$, and $\mu_n(F)\to \mu(F)$ in probability for every continuous function $F$ on $K$, then $\mu_n$ is said to converge to $\mu$ in probability in the weak* sense. 

\begin{theorem}[General construction] 
	\label{thm:const_gen}
	Suppose that $X$ is a log-correlated field as in \cref{sec:general_fields} and that $(X_\eps)_{\eps\ge 0}$ is a convolution approximation to $X$ (defined on the same probability space). Then there exists a non-trivial measure $\mu'$ on $D$ such that for any $K\subset D$ compact, \begin{equation}\label{eqn:const_thm_der} \mu_\eps'|_K\to \mu'|_K \end{equation} in probability in the weak* sense, along any sequence $\eps_n$ converging to $0$. The measure $\mu'$ does not depend on the choice of mollifier used.
	
	 Furthermore, for any such $K$, \begin{equation}(\sqrt{\log(1/\eps)}\,\mu_\eps^{\gamma_c})|_K\to (\sqrt{\frac{2}{\pi}} \, \mu')|_K\end{equation} in probability $\eps\to 0$ (in the same sense).
\end{theorem}

\begin{rmk}\label{rmk:conv_starscale}
	The same result holds for $\star$-scale invariant measures, with the analogous approximations $\sqrt{t}\mu^{\gamma_c}_t$ and $\mu_t'$ described in \cref{sec:renorm}. In this case the convergence of $\mu_t'$ is actually a.s.\ for the topology of weak convergence of measures (due to the underlying martingale structure), while the convergence of $\sqrt{t}\mu^{\gamma_c}_t$ only holds in probability. It should be expected that the latter convergence \emph{cannot} be lifted to almost sure convergence. Indeed in the branching random walk setting of this result, \cite{AiSh}, it is shown that there almost surely exists a sequence of $t$ along which convergence does not hold (although this has not been proven for Gaussian multiplicative chaos measures). 
	
\end{rmk}
\subsubsection{History}
As already hinted at, the history behind this construction really begins with the work of \cite{BK04} concerning convergence of the so-called derivative martingale for the branching random walk, and the subsequent work \cite{AiSh} on the Seneta--Heyde normalisation. Motivated by these ideas, the picture for critical Gaussian multiplicative chaos has gradually evolved towards \cref{thm:const_gen}. The timeline is roughly as follows.

\begin{enumerate}
	\item[(2014)] Duplantier, Rhodes, Sheffield and Vargas \cite{DRSV14one,DRSV14two} construct critical Gassian multiplicative chaos for $\star$-scale invariant fields as in \eqref{eqn:ssi}. This is defined as the almost sure limit of $\mu_t'$ as $t\to \infty$, where $\mu_t'$ is as described in \cref{ex:BMDM}, making use of the fact that approximate masses of fixed sets are martingales in this set-up. The authors show further that $\sqrt{t}\mu^{\gamma_c}_t$ (again defined using the $\star$-scale cut-off approximations) converges in probability to $(\sqrt{\pi/2}) \mu'$ as $t\to \infty$. Finally, they allow some relaxation on the need for $k$ to have compact support. This allows them, among other things, to consider and prove the same results for the planar Gaussian free field; see \cite{DRSV14two}.

	\item[(2015)] Huang, Rhodes and Vargas \cite{HRV18} use a comparison argument, in the case of the planar GFF and the GFF  on the unit circle \eqref{eqn:gff_cov_circle}, to show that in the Seneta--Heyde normalisation \eqref{eqn:sh_norm} with convolution approximations to the field, the same limit $(\sqrt{2/\pi})\mu'$ defined in \cite{DRSV14one, DRSV14two} can be obtained. 
	\item[(2016)] At essentially the same time, Junnila and Saksman prove a general comparison theorem \cite{JS17} that allows for comparison between different approximations in the Seneta--Heyde normalisation. This shows that \emph{if} the limit \eqref{eqn:sh_norm} (or its analogue for a non-convolution approximation) exists for one approximation scheme, it also exists for other comparable ones, and the limiting measures will be the same. In particular, it extends the result of \cite{DRSV14two} (the Seneta--Heyde normalisation) to a wider class of approximations (e.g., convolution) for $\star$-scale invariant fields and the Gaussian free field.
	\item[(2017)] For the other, derivative, normalisation scheme \eqref{eqn:der_norm_eps}, \cite{Pow18chaos} again considers $\star$-scale invariant fields and the planar GFF. The article shows that if the limit \eqref{eqn:sh_norm} in the Seneta--Heyde normalisation for convolution approximations exists, then it also exists in the derivative normalisation \eqref{eqn:der_norm_eps}, and the two limits agree up to the constant $\sqrt{2/\pi}$. The proof can be extended to the GFF on the unit circle, as explained in \cite{APS18two}.
	
So at this point, \cref{thm:const_gen} is known for $\star$-scale invariant fields, the GFF on the circle, and the planar GFF. However, existence of limits is not known in \emph{any} scheme for more general fields, so the comparison theorems of \cite{JS17,Pow18chaos} cannot be applied.

	\item[(2018)] In an attempt understand \eqref{eqn:der_norm} (more on this in the next section) Aru, Sep\'{u}lveda and the current author develop a new construction of GMC measures for the planar Gaussian free field \cite{APS18one}, using the approximation \eqref{eqn:cle_approx}. With this approximation, they show convergence to the critical measure in both the Seneta--Heyde and derivative normalisations. 

	\item[(2019)] Finally, Junnila and Saksman return to the scene, now together with Webb \cite{JSW19}, and armed with a clever decomposition theorem. This essentially says that you can write any log-correlated field as in \eqref{eqn:kx} as a sum of a $\star$-scale invariant field and an independent continuous Gaussian field. As a consequence, together with previous results, one obtains \cref{thm:const_gen}.
\end{enumerate}

\subsubsection{Ideas behind the proof of \cref{thm:const_gen}}

(I)\emph{ Convergence in the derivative normalisation for $\star$-scale invariant fields.} \\

Let us consider the approximate measures $\mu_t'$ as in \cref{ex:BMDM}, when the underlying field $X$ is $\star$-scale invariant and $(X_t)_{t\ge 0}$ are the $\star$-scale cut off approximations to $X$ with covariances \eqref{eqn:starcut}. By standard results concerning convergence of measures, see for example \cite{DRSV14one}, it is enough to show that  for any $A\subset \R^d$ compact,
\begin{equation}
\label{eqn:der_conv_mass}
\mu_t'(A)= \int_A (-X_t(z)+\gamma_c \var(X_t(x))\e^{\gamma_c X_t(x)-\frac{\gamma_c^2}{2}t}\, dx
\end{equation}
has an almost sure limit as $t\to \infty$. Recall that:
\begin{itemize}
	\item[-]for any $x\in \R^d$, $(X_t(x); t\ge 0)$ has the law of a standard linear Brownian motion;
	\item[-]$(X_t(x);\, t\ge 0)$ is a martingale with respect to the filtration $(\cF_t;\, t\ge 0)$; where $\cF_t=\sigma((X_s(x);\, x\in A, s\in [0,t] ))$.
\end{itemize}

Exchanging integral and expectation means that $\mu_t'(A)$ is itself a martingale for the filtration $(\cF_t;\, t\ge 0)$. This is  good news for taking limits, but on the other hand, $\mu_t'(A)$ need not be positive and  so convergence is not guaranteed. Note that this point really does deserve careful consideration, since the aim is to construct a positive measure in the end.

In an attempt to get around this, one can turn to the measures 
\begin{equation}
\label{eqn:Dtb}
D_t^\beta(dx)=d_t^\beta(x) dx=(-X_t(x)+\gamma_c t+\beta)\e^{\gamma_ cX_t(x)-\frac{\gamma_c^2}{2}t}\I_{A_t^\beta(x)} \, dx \text{ for } \beta>0,
\end{equation}
where $A_t^\beta(x)$ defined in \eqref{def:dtbatb} is the event that $(-X_s(x)+\gamma_c s+\beta)$ stays non-negative for $0\le s\le t$. The advantage is that $D_t^\beta(A)$ is then non-negative (by definition). Moreover, since $d_t^\beta(x)$ is a martingale with respect to $(\cF_t; t\ge 0)$ for each $x$ (as explained in \cref{ex:BMDM}), $D_t^\beta(A)$ is also a martingale.

Therefore, for any $\beta>0$, $D_t^\beta(A)$ has an almost sure positive limit $D^\beta(A)$ as $t\to \infty$.  Furthermore, it is clear that $D^\beta(A)$ must be increasing in $\beta$ and so also have an almost sure positive limit $D(A)$ as $\beta \to \infty$. But what does this say about the convergences of $\mu_t'(A)$? The key is that by \eqref{eqn:extrema_basic}, 
\begin{equation}
\P( \exists \beta\in (0,\infty) \text{ such that } \mu_t'(A)=D_t^\beta(A)-\beta \mu_t^{\gamma_c}(A) \text{ for all } t)=1,
\end{equation}
while on the other hand by \cref{rmk:conv_to_0_star}, $\beta \mu_t^{\gamma_c}(A)\to 0$ almost surely. Combining all of this,
it follows that $\mu_t'(A)$ converges almost surely to $D(A)=:\mu'(A)$ as $t\to \infty$.

So all that is left to complete the picture in this setting is to show that $\mu'(A)$ is non-trivial. Observe that for this, it suffices to show that $D_t^\beta(A)$ is uniformly integrable for each $\beta$. Indeed, if this is the case then the martingale $D_t^\beta(A)$ will converge in $L^1(\P)$ for each $\beta$, and the limit $D^\beta(A)$ will have expectation $\beta |A|$  (since $\E(d_t^\beta(x)=\beta |A|$ for every $x$).  $D^\beta(A)$ will therefore be non-trivial for each $\beta$, and increasing as $\beta\to \infty$, meaning that $\mu'(A)$ is certainly non-trivial.

Notice that $\mu_t'(A)$ then necessarily has infinite expectation, since it must be greater than $\beta|A|$ for any $\beta$. In particular, the convergence $\mu_t'(A)\to \mu'(A)$ \emph{does not} hold in $L^1(\P)$. \\

\noindent (II)\emph{ Spine decomposition/rooted measures.}\\

 To prepare for the proof of uniform integrability (and later proofs), it is necessary to now discuss an analogue of the ``spine decomposition'' associated with derivative measures and martingales. This is the natural extension of \cref{sec:spine}. 

Just as before, fix $A\subset D$ compact and define $\mathbb{P}^*$ to be the usual $\mathbb{P}$ law of $(X_t(x);\, x\in \R^d, t\ge 0)$ together with a random point $x^*$ chosen proportionally to Lebesgue measure in $A$. Then from  \cref{ex:BMDM}, it follows that $d_t^\beta(x^*)$ is a martingale with respect to the filtration $(\cF_t^*; \, t\ge 0)$ where $\cF_t^*$ for each $t$ is the $\sigma$-algebra generated by $x^*$ together with $(X_s(x);\, x\in A, s\in [0,t])$. Since the expectation of $d_t^\beta(x^*)$ is $\beta$, it is possible to define a measure $\mathbb{Q}^*$ via 
\begin{equation}\label{eqn:starrn}
\frac{d\Q^{\beta,*}}{d\P^*}|_{\cF_t^*}:= \frac{d_t^\beta(x^*)}{\beta} \; \; \forall t.
\end{equation}

Again observe that 
\begin{equation}\label{def:qbeta}
\frac{d\Q^{\beta,*}}{d\P^*}|_{\cF_t}:= \frac{D_t^\beta(A)}{\beta|A|} \; \forall t,
\end{equation}
so if $\mathbb{Q}^\beta$ denotes the marginal law of $(X_t(x);\, x\in \R^d, t\ge 0)$ under $\mathbb{Q}^{\beta,*}$, then 
the Radon--Nikodym derivative $(d\Q^\beta/d\P)|_{\cF_t}$ is proportional to $D_t^\beta(A)$ for every $t$. 

One can also easily compute that 
\begin{equation}\label{eqn:xgivenf}
\Q^{\beta,*}(\{x^*\in B\}\, | \,\cF_t)=\frac{D_t^\beta(B)}{D_t^\beta(A)}, \; \forall B\subset A;
\end{equation}
in other words, given $\cF_t$ and under $\Q^{\beta,*}$, $x^*$ is chosen according to $D_t^\beta$ in $A$. On the other hand, the $\mathbb{Q}^{\beta,*}$-marginal law of $x^*$ is just the uniform distribution in $A$. 

Finally it follows from the discussion in \cref{ex:BMDM}, that given $x^*$, the conditional $\Q^{\beta,*}$ law of $(-X_t(x^*)+\gamma_c t + \beta; t\ge 0)$ is that of a 3d-Bessel process started from $\beta$.
\\

\noindent \emph{Conclusion of }(I)

\begin{proof}[Proof that $D_t^\beta(A)$ is uniformly integrable]
	By definition of $\Q^\beta$, showing that $D_t^\beta(A)$ is uniformly integrable amounts to showing that $\Q^\beta(D_t^\beta(A)>K)\to 0$ as $K\to \infty$, uniformly in $t$ (cf. the discussion below \cref{lem:conv_to_0}). They key idea is to replace $\Q^\beta$ by $\Q^{\beta,*}$ in this probability (since they give exactly the same mass to the event in question) and then decompose according to whether the spine $(-X_t(x^*)+\gamma_c t +\beta)$ behaves reasonably, or not. Let us assume without loss of generality that $A\subset [0,1]^d$.
	
	Recall that $(-X_t(x^*)+\gamma_c t +\beta)$ has the law of a 3d-Bessel process under $\Q^{\beta,*}$, which at time $t$ is typically of order $\sqrt{t}$. In fact, one can be much more precise. It follows from \cite{Motoo}, that if 
	\begin{equation}
	E_{t}^{R}:=\{R(1+\sqrt{s\log(1+s)})\le (-X_s(x^*)+\gamma_c s +\beta) \le R\frac{\sqrt{s}}{\log(2+s)^2} \; \forall s\in[0,t]\}
	\end{equation}
then $\Q^{\beta,*}(E_\infty^R)\to 1$	as $R\to \infty$. Since by Markov's inequality, one can write $\Q^{\beta,*}(D_t^\beta(A)\ge K)\le \Q^{\beta,*}(E_t^R)+K^{-1}\E_{Q^{\beta,*}}(\I_{E_t^R}D_t^\beta(A))$ for any $R$, it therefore suffices to show that for any fixed $R$:
\begin{equation}\label{eqn:ui_bound}
\sup_t\E_{\Q^{\beta,*}}(\I_{E_t^R}D_t^\beta(A))<\infty.
\end{equation}
So let us now show this. For each $y\in A$, write $t_0^*(y)=-\log(|x^*-y|)$. The important fact is that under $\P^*$ and given $(x^*, (X_s(z); s\le t_0^*(y), z\in A))$, 
\begin{equation}\label{eqn:cond_ind} d_t^\beta(y)-d_{t_0^*(y)}^\beta(y) \text{ and } d_t^\beta(x^*)-d_{t_0^*(y)}^\beta(x^*) \text{ are conditionally independent} \end{equation} (and both with conditional expectation $0$). Indeed this follows since $x^*$ is chosen independently of $(X_s(x);\, x\in A, s\ge 0)$ under $\P^*$, since $d_t^\beta(x)$ is a martingale for each $x\in A$, and by the nice decorrelation property of the approximations $X_s$ to $X$: recall the discussion below \eqref{eqn:starcut}. Thus,
\begin{align*}
& \E_{\qbs}(\I_{E_t^R}D_t^\beta(A)) &&  \\
& = \beta^{-1}\int_A \E_{\P^*}( \I_{E_t^R} d_t^\beta(x^*)d_t^\beta(y)) \, dy  && \textit{by Fubini and \eqref{eqn:starrn}} \\
& \le \beta^{-1} \int_A \E_{\P^*}(\I_{E_{t_0^*(y)}^R}d_t^\beta(x^*)d_t^\beta(y) ) dy && \textit{since $E_{s}^R$ is increasing }\\
& = \beta^{-1} \int_A \E_{\P^*}(\I_{E_{t_0^*(y)}^R}d_{t_0^*(y)}^\beta(x^*)d_{t_0^*(y)}^\beta(y) ) dy && \textit{by \eqref{eqn:cond_ind}} \\
& \lesssim \E_{\P^*}(\int_A R\frac{\sqrt{t_0^*(y)}}{\log(2+t_0^*(y))^2} \e^{-\gamma_c R\sqrt{t_0^*(y)\log(1+t_0^*(y))}}\e^{dt_0^*(y)} \, d^\beta_{t_0^*(y)} \, dy) && \textit{ by definition of  $E_s^R$}\\
 &\lesssim \int_{B(0,1)} \frac{\sqrt{\log(|w|^{-1})}}{\log(2+\log(|w|^{-1}))^2} \e^{-\sqrt{2d}R\sqrt{\log(|w|^{-1})\log(1+\log(|w|^{-1}))}}\e^{d\log(|w|^{-1})}  \, dw &&
\end{align*}
where the implied constants in the final two inequalities depend only on the fixed quantities $A,R,\beta$. The final line follows from the fact that $x^*$ is chosen according to Lebesgue measure in $A$ under $\P^*$.

So, all that remains to check is that this integral is finite. This is easily verified by changing to hyper-spherical coordinates in $\R^d$, which nicely cancels the blowing up term $\e^{d\log(|w|^{-1})}$. Since the rest is very well behaved as $|w|\to 0$, it is not hard to conclude: see \cite{DRSV14one} for a step-by-step proof.

\end{proof}

\noindent (III) \emph{The Seneta--Heyde normalisation for $\star$-scale invariant fields.}\\

 The convergence of $\sqrt{t}\mu_t'$ as in \cref{rmk:conv_starscale} follows from a special case of the next lemma. 

\begin{lemma}\label{lem:sh}
Take any $A\subset \R^d$ compact. Then if $F$ is continuous, positive and bounded in $A$, the convergence 
	\begin{equation}\label{eqn:sh}
	\sqrt{t} \; \frac{\int_A \e^{\gamma_c X_t(x)-\frac{\gamma_c^2}{2}t} F(\frac{- X_t(x)+{\gamma_c}t}{\sqrt{t}}) \, dx }{\mu_t'(A)} \to \sqrt{\frac{2}{\pi}}\, \E(F(R_1))
	\end{equation}
 holds in probability as $t\to \infty$, where $R_1$ is a Brownian meander at time $1$.
\end{lemma}

In particular, taking $F\equiv 1$, it follows that 
\begin{equation}
\frac{\sqrt{t} \mu_t^{\gamma_c}(A)}{\mu_t'(A)}\to \sqrt{2/\pi}
\end{equation}
in probability as $t\to \infty$. Together with fact that $\mu_t'(A)\to \mu'(A)$ almost surely as $t\to \infty$, this implies that $\sqrt{t}\mu_t^{\gamma_c}(A)$ converges in probability to $(\sqrt{2/\pi}) \mu'(A)$. As before, this suffices to show convergence in probability of the measures $\sqrt{t}\, \mu_t^{\gamma_c}$ to $(\sqrt{2/\pi}) \mu'$. 

\begin{proof}[Proof of \cref{lem:sh}]

The backbone of this argument is based on its branching random walk analogue in \cite{AiSh}, although there are specific parts of the analysis that are both easier and harder in the continuum setting. The lemma with $F\equiv 1$ is proved in \cite{DRSV14two}, but it is essentially identical for general continuous and bounded $F$.

 Write 
\begin{equation} F_t^\beta = \int_A \1_{A_t^\beta(x)} \e^{\gamma_c X_t(x)-\frac{\gamma_c^2}{2}t} F(\frac{- X_t(x)+\gamma_ct}{\sqrt{t}}) \, dx.\end{equation} Recalling the definition \eqref{def:qbeta}, the idea is to show that 
\begin{equation}\label{eq:fm}
\sqrt{t}\, \mathbb{E}_{\Q^\beta}(\frac{F_t^\beta}{D_t^\beta(A)})=\sqrt{\frac{2}{\pi}}\E(F(R_1))+o(1)
\end{equation}	and 
\begin{equation}\label{eq:sm}
t \, \mathbb{E}_{\Q^\beta}((\frac{F_t^\beta}{D_t^\beta(A)})^2)\le \frac{2}{\pi}\E(F(R_1))^2+o(1)
\end{equation}
as $t\to \infty$. By Markov's inequality, \eqref{eq:fm} and \eqref{eq:sm} imply that for every $\beta$, $\sqrt{t} (F_t^\beta/D_t^\beta(A))$ converges to $\sqrt{2/\pi}$ in $\Q^\beta$-probability as $t\to \infty$. This means by \eqref{def:qbeta}, that for any $\beta, \delta>0$:
\begin{equation}
\P \left( D_t^\beta(A)^{-1} \I_{\{|\sqrt{t} (F_t^\beta/D_t^\beta(A))-\sqrt{2/\pi}|>\delta\}}\right) \to 0 \text{ as } t\to \infty.
\end{equation}
On the other hand, by \eqref{eqn:extrema_basic} and the fact that $\lim_{\beta\to \infty} \lim_{t\to \infty} D_t^\beta(A)=\mu'(A)$ is non-trivial, one can take $\beta$ large enough and $\eta$ small enough that the events \begin{equation}\left\{\frac{F_t^\beta}{D_t^\beta(A)}=\frac{\int_A \e^{\gamma_c X_t(x)-\frac{\gamma_c^2}{2}t} F(\frac{- X_t(x)+{\gamma_c}t}{\sqrt{t}}) \, dx }{\mu_t'(A)+\beta \mu_t^{\gamma_c}(A)} \;\; \forall t\ge 0\right\} \text{ and } \left\{D^\beta(A)=\lim_{t\to \infty} D_t^\beta(A)\ge \eta\right\} \end{equation}
occur with probability arbitrarily close to one. Since $\mu_t^{\gamma_c}(A)\to 0$ as $t\to \infty$, \cref{lem:sh} follows straightforwardly from these observations.
  Let us omit the technical details, that can be found in \cite[Appendix B]{APS18one} (or in \cite{DRSV14two} for $F\equiv 1$, which is almost identical).\\
  
  For the first moment estimate \eqref{eq:fm}, the argument is pretty neat. Simply write
  \begin{equation}\label{eq:fmint} \sqrt{t}\, \mathbb{E}_{\Q^\beta}(\frac{F_t^\beta}{D_t^\beta(A)}) =\sqrt{t}\frac{\E_{\P}(F_t^\beta)}{\beta|A|}= \frac{1}{|A|}\int_A \frac{\sqrt{t}}{\beta}\E_\P (\I_{A_t^\beta(x)} \e^{\gamma_c X_t(x)-\frac{\gamma_c^2}{2}t} F(\frac{- X_t(x)+{\gamma_c}t}{\sqrt{t}})) \, dx \end{equation}
  and observe that the integrand on the right-hand side does not depend on $x$. Indeed the law of $X_t(x)$ under $\P$ is that of a standard linear Brownian motion for each $x$. Moreover by Girsanov's theorem, after changing measure with the Radon--Nikodym martingale $\exp(\gamma_c X_t(x)-(\gamma_c^2/2)t)$, the law of $(-X_t(x)+\gamma_c t)$ becomes just that of a standard linear Brownian motion $B$. Writing $\mathbf{P}$ for the law of $B$, the right-hand side of \eqref{eq:fmint} can be expressed as 
  \begin{equation}
\frac{\sqrt{t}}{\beta}\E_{\mathbf{P}}(F(t^{-1/2} B_t)\I_{\{\inf_{s\in[0,t]}B_s \ge -\beta\}})  \overset{t\to \infty}{=}\sqrt{\frac{\pi}{2}}\E_{\mathbf{P}}(F(t^{-1/2} B_t)\, | \, \inf_{s\in[0,t]} B_s \ge -\beta)+o(1),  \end{equation}
where the equality follows since  $\mathbf{P}(\inf_{s\in[0,t]}B_s \ge -\beta) \sim \sqrt{\pi\beta^2/2t}$ as $t\to \infty$. \eqref{eq:fm} then follows by (scaling and the) characterisation \cite{DIM77} of a Brownian meander on $[0,1]$ as the limit as $\eps\to 0$ of Brownian motion conditioned to stay above $-\eps$ on $[0,1]$.\\

 The second moment bound \eqref{eq:sm} is rather more complicated, but is based around one clever trick. That is, to notice that 
 \begin{equation}\label{eqn:ltb}
 \frac{F_t^\beta}{D_t^\beta(A)}:=\Q^{\beta,*}(L_t^\beta(x^*)\, | \, \cF_t)\; \text{ where } \; L_t^\beta(x^*):= \frac{F(t^{-1/2}(- X_t(x^*)+\gamma_c t))}{-X_t(x^*)+\gamma_c t+ \beta}.
 \end{equation} 
 Indeed this just follows from the definition of $(F_t^\beta, D_t^\beta(A))$ and the expression \eqref{eqn:xgivenf}  describing the conditional density of $x^*$ given $\cF_t$.
 
 This means, by the tower property of conditional expectation,  that the left-hand side of \eqref{eq:sm} can be rewritten as \begin{equation}\label{eqn:sm_trick}
 t\, \E_{\Q^{\beta,*}}(\frac{F_t^\beta}{D_t^\beta(A)} L_t^\beta(x^*)).
 \end{equation}
Recall the aim is to bound this by $ (2/\pi)\E(F(R_1))^2+o(1)$. The basic strategy is to say that the behaviour of the field close to the point $x^*$ will not have a significant effect on the large-$t$ behaviour of $F_t^\beta/D_t^\beta(A)$: this is, essentially, because $\lim_t D_t^\beta$ does not have any atoms. As a result, one can (roughly speaking) factorise the expectation in \eqref{eqn:sm_trick} into the two parts $\E_{\Q^{\beta,*}}(F_t^\beta/D_t^\beta(A))$ and $\E_{\Q^{\beta,*}}(L_t^\beta(x^*))$, while incurring only an $o(1)$ error in $t$. Applying the first moment asymptotic \eqref{eq:fm} then yields the conclusion. This argument, although simple in principle, becomes quickly quite technical. Below the technical details are again omitted, and the reader is referred to \cite{DRSV14two} for a more thorough treatment. 
 
The first step is to reduce the problem to showing that
 \begin{equation}\label{eqn:sm_reduced}
 t \, \E_{\Q^{\beta,*}}(\frac{F_t^\beta}{D_t^\beta(A)}\I_{E_t} L_t^\beta(x))\le \frac{2}{\pi }\E(F(R_1))^2+o(1)
 \end{equation}
as $t\to \infty$, for any sequence $E_t$ of events with $\qbs(E_t)\to 1$ as $t\to \infty$. The advantage of this is that situations where $X_t(x^*)$ behaves abnormally can be ignored. The reduction itself follows from some straightforward bounds: it is shown in \cite{DRSV14two} with $F\equiv 1$, but exactly the same proof works when $F$ is bounded.

Next, setting 
\begin{equation}h_t=t^{5/12}\end{equation} (note that this is smaller than the typical growth of a 3d-Bessel process but still blows up with $t$) one can decompose 
\begin{equation} D_t^\beta(A) = \int_{A\setminus B(x^*,\e^{-h_t})} d_t^\beta(x)dx + \int_{B(x^*,\e^{-h_t})}  d_t^\beta(x) \, dx:=\tilde D_t^\beta+\int_{B(x^*,\e^{-h_t})}  d_t^\beta(x) \, dx\end{equation}
and similarly for $F_t^\beta$, thus defining $\tilde{F}_t^\beta$. Then another technical argument, basically using the fact that $\e^{-h_t}$ decays sufficiently fast with $t$ to say that $(F_t^\beta- \tilde{F}_t^\beta)$ does not contribute to the left-hand side of \eqref{eqn:sm_reduced} in the limit, reduces the proof of \eqref{eqn:sm_reduced} to showing that 
\begin{equation}\label{eqn:sm_final}
t\, \E_{\qbs}(\frac{\tilde F_t^\beta}{\tilde D_t^\beta} L_t^\beta(x^*) \I_{E_t})\le \frac{2}{\pi} \E(F(R_1)^2)+o(1) \end{equation}
as $t\to \infty$, where  
\begin{equation}E_t:=\{\tilde{F}_t^\beta\le \tilde{D}_t^\beta \}\cap \{-X_{h_t}(x^*)+\gamma_c {h_t}+\beta \in [h_t^{1/3},h_t]\}.\end{equation} Since $-X_{s}(x^*)+\gamma_c s+\beta$ evolves as a 3d-Bessel process under $\qbs$, it is straightforward to show that $\qbs(E_t)\to 1$ as $t\to \infty$. Again this is shown in full detail in \cite{DRSV14two} with $F\equiv 1$, and the same argument works with bounded $F$. 

The final trick is to condition on $\cG_t^*$: the $\sigma$-algebra generated by $x^*$ and $(X_s(x);\, s\le h_t, x\in A)$. The point is that, by the decorrelation property of $(X_t;\, t\ge 0)$ (see after \eqref{eqn:starcut}) and the definition of ($\tilde{F}_t^\beta$, $\tilde{D}_t^\beta$), the event $E_t$ is measurable with respect to $\cG_t^*$, while $(\tilde{F}_t^\beta/\tilde{D}_t^\beta)$ and $L_t^\beta(x^*)$ are conditionally independent given it.
Thus, 
\begin{equation} t\, \E_{\qbs}(\frac{\tilde F_t^\beta}{\tilde D_t^\beta} L_t^\beta(x^*) \I_{E_t})= \E_{\qbs}(\I_{E_t} \sqrt{t} \E_{\qbs} ( L_t^\beta(x^*)\, | \, \mathcal{G}^*_t ) \sqrt{t} \E_{\qbs}( \frac{\tilde{F}_t^\beta}{\tilde D_t^\beta}\, | \, \cG^*_t )).\end{equation}
This is the ``factorisation'' stage mentioned in the proof outline. 

To conclude, the definition of $E_t$ and the first moment argument, again using \eqref{eqn:ltb}, gives that  \begin{equation}\I_{E_t}\E_{\qbs}(L_t^\beta(x^*)|\mathcal{G}_t^*)\le \sqrt{\frac{2}{\pi (t-h_t)}}\E(F(R_1)) \text{ almost surely}, \end{equation}
 while a slightly more in-depth argument gives that 
 \begin{equation}
 \sqrt{t}\E_{\Q^{\beta,*}}(\I_{E_t} \frac{\tilde{F}_t^\beta}{\tilde{D}_t^\beta})\le \sqrt{t} \E_{\Q^{\beta,*}}(\frac{F_t^\beta}{D_t^\beta(A)})+o(1) \end{equation}
(this is where the fact that $\tilde{F}_t^\beta \le \tilde D_t^\beta$ on $E_t$ is used, although it is hidden at this level of exposition). Putting these together, since $\sqrt{(t/t+h_t)}=1+o(1)$ as $t\to \infty$, provides \eqref{eqn:sm_final} and thus completes the proof.
  
\end{proof}

\noindent (IV) \emph{ Comparison arguments } \\

Recall that in \cite{JS17}, the authors prove that in the Seneta--Heyde normalisation, convergence for one approximation to the field implies convergence in other comparable approximations.
More precisely, they show the following.

\begin{theorem}[\cite{JS17}]\label{thm:JS}
Suppose that $(X_n,\tilde{X}_n;\, {n\ge 0})$ are two sequences of centred Gaussian fields on $D\subset\R^d$ compact, with covariance kernels $(K_n,\tilde{K}_n; \, n\ge 0)$ such that \begin{equation}(x,y)\mapsto \sqrt{\E(X_n(x)-X_n(y))^2} \text{ and } (x,y)\mapsto \sqrt{\E(\tilde{X}_n(x)-\tilde{X}_n(y))^2}
\end{equation} are $\alpha$-H\"{o}lder continuous for some $\alpha>0$. Suppose further that
\begin{equation}
\sup_{x,y\in D}|K_n(x,y)-\tilde{K}_n(x,y)|<\infty \; \forall n\ge 1 \text{ and } \lim_{n\to \infty} \sup_{|x-y|>\delta}|K_n(x,y)-\tilde{K}_n(x,y)|=0 \; \forall \delta>0.
\end{equation}
Finally, assume that $(\rho_n;\ n\ge 0)$ is a sequence of non-negative Radon reference measures, and that $\tilde{\mu}_n(dx):=\exp(\tilde{X}_n(x)-\frac{1}{2}\E(\tilde{X}_n(x)^2)) \rho_n(dx)$ converges in distribution to an a.s.\ non-atomic random measure $\tilde{\mu}$ on $D$ as $n\to \infty$. Then ${\mu}_n(dx)$ (defined analogously) converges in distribution to the same random measure $\tilde{\mu}$. 
\end{theorem} 

There are also conditions provided \cite[Theorem 4.4]{JS17} for the analogous result to hold with convergence in distribution replaced by convergence in probability.\\

Note the freedom in allowing the reference measures $\rho_n$ to depend on $n$ here. This is what makes the theorem applicable to approximations of Gaussian multiplicative chaos in the Seneta--Heyde normalisation. For  example, one can let $X_n,\tilde{X}_n$ be $\gamma_c$ times some convolution approximations to a log-correlated field at level $\eps_n$, and set $\rho_n(dx):=\sqrt{\log(1/\eps_n)} \, dx$. 

See \cite[Section 5]{JS17} for proofs that the conditions of this theorem are satisfied when comparing:

\begin{itemize}
	\item different convolution approximations of general log-correlated fields (Corollary 5.2); 
	\item convolution approximations vs. $\star$-scale cut-off approximations of $\star$-scale invariant fields (Lemma 5.6);
	\item convolution approximations vs. circle averages of the planar Gaussian free field (Lemma 5.7).
\end{itemize} 

\begin{proof}[Sketch proof (of \cref{thm:JS})]
	The first step is a fairly classical argument (using Kahane's inequalities) to show that tightness of the sequence of measures $(\tilde{\mu}_n;\, n\ge 1)$ implies tightness of $(\mu_n; \, n\ge 1)$. Thus, $(\mu_n; \, n\ge 1)$ has subsequential limits in the space of measures, and it is only necessary to prove that any such limit $\mu$ must be equal to $\tilde{\mu}$. For this, it suffices to prove that for any non-negative continuous function $f$ on $D$ and any non-negative, bounded, continuous concave function $\varphi$, that 
	\begin{equation}\label{eq:JSproof1}
	\E(\varphi(\int f(x) \mu(dx)))=\E(\varphi(\int f(x) \tilde{\mu}(dx)).
	\end{equation} 
	The idea is to use Kahane's concavity inequality (just take $F=-F$ in \cref{thm:kahane}), to verify that the above equation is satisfied with both $\le$ and $\ge$ in place of $=$.
	
	To show the $\le$ version (the other case following from the symmetric argument) suppose for ease of notation that $\mu_n\to \mu$ as $n\to \infty$. 
	The idea is to define an auxiliary field $Y$, continuous and independent of $(X_n; \, n\ge 1)$, so that $X_n+Y$ has covariance dominating the covariance of $\tilde{X}_n$ pointwise for all large enough $n$, but also so that $Y$ is very close to being totally decorrelated. The first of these properties means that \begin{equation}\label{eq:JSproof2}\E(\varphi(\int f(x) \e^{Y(x)-\frac{1}{2}\E(Y(x)^2)} \mu(dx)))\le \E(\varphi(\int f(x) \tilde{\mu}(dx)))\end{equation} by Kahane, and the second means that the left-hand side of \eqref{eq:JSproof2} very close to the left-hand side \eqref{eq:JSproof1}. 
	
	More precisely, Lemma 3.5 in \cite{JS17} shows that one can construct a sequence of fields $(Y_\eps;\, \eps>0)$ so that \eqref{eq:JSproof2} holds with $Y\leftrightarrow Y_\eps$ for every $\eps$, and so that \begin{equation}\label{eqn:Yeps}\E(|\int_D \e^{Y_\eps(x)-\frac{1}{2}\E(Y_\eps(x)^2)} \lambda(dx)-\lambda(D)|^2)\lesssim \eps^2 \lambda(D)^2+\iint_{|x-y|<2\eps} \lambda (dx) \lambda (dy)\end{equation} for \emph{any} positive measure $\lambda$ on $D$. Note that the right-hand side is small as long as $\lambda$ isn't atomic. In particular since $\mu$ must be a.s.\ non-atomic (this is Lemma 3.3 in \cite{JS17}, which is again proved using Kahane's inequality), \eqref{eqn:Yeps} can be applied with $\lambda=\mu$
to deduce that \begin{equation}\label{eq:JSproof3} \int f(x) \e^{Y_\eps(x)-\frac{1}{2}\E(Y_\eps(x)^2)} \mu(dx)\to \int f(x) \mu (dx)\end{equation} almost surely as $\eps \to 0$ (this last step uses a Borel--Cantelli argument). Combining with \eqref{eq:JSproof2} yields the conclusion. 
	\end{proof}

The next comparison argument to appear after this was in \cite{Pow18chaos}, whose main theorem is:

\begin{theorem}[\cite{Pow18chaos}]
 For $\star$-scale invariant fields and the planar Gaussian free field, convergence for convolution approximations in the Seneta--Heyde normalisation implies convergence in the derivative normalisation. 
\end{theorem}

The proof of this theorem goes along similar lines to the proof of \cref{lem:sh} above, though dealing with quite different technicalities (that will not be discussed here). 

Finally, the theorem that pulls everything together to reach the conclusion of \cref{thm:const_gen} is the following, much more general result from \cite{JSW19}. 

\begin{theorem}[\cite{JSW19}]
	\label{thm:decomp}
	Let $X_1,X_2$ be two centred Gaussian fields, almost surely lying in $H^{\alpha}_{\text{loc}}(D)$ for some $\alpha>0$ and some domain $D\subset \R^d$. Let $C_1,C_2$ be the covariance kernels of $X_1, X_2$, and assume that $C_1,C_2\in L_{\text{loc}}^1(D\times D)$ while for some $\eps>0$, $C_1-C_2\in H_{\text{loc}}^{d+\eps}(D\times D)$. Then for any subdomain $D'$ compactly contained in $D$, it is possible to construct $(X_1',X_2',G)$ on a common probability space, such that $G$ is a.s.\ H\"{o}lder continuous on $D'$, $X_1',X_2'$ have the same marginal laws as $X_1,X_2$, and
	\begin{equation}
	X_1'=X_2'+G \text{ almost surely on } D'.
	\end{equation}
	
\end{theorem}

The proof of this theorem is beyond the scope of this survey. But, as mentioned previously, the result is the ability to decompose any log-correlated field $X$ as in \eqref{eqn:kx}, as a sum of a $\star$-scale invariant field and an independent H\"{o}lder continuous Gaussian field. Since \cref{thm:const_gen} is known for $\star$-scale invariant fields, this finally (with a little work) implies the general result. 

Unsurprisingly this is just one application of the decomposition theorem: see \cite{JSW19} for a much broader discussion.

\subsection{Properties of critical measures}
 
This section provides a brief survey, without proofs, of some important properties of critical Gaussian multiplicative chaos.
Unless stated otherwise, $\mu'$ will be the critical chaos measure for a general log-correlated field as in \eqref{eqn:kx}.

\begin{theorem}[Moments, \cite{DRSV14two}]
For any $A$ non-empty, bounded and open, $\mu'(A)$ has finite moments of every order $q\in(-\infty,1)$. It does not have moments of order $1$.
Moreover, for any $q<1$ there exists a constant $C_q$ (that may depend on $K_X$) such that for any $A\subset D$ open and bounded
\begin{equation}\label{eqn:mf}
\E[\mu'(rA)^q]\overset{r\to 0}{\asymp} C_q r^{2dq-dq^2}.
\end{equation}
\end{theorem} 

For $\star$-scale invariant fields, the fact that $\mu'(A)$ cannot have a finite moment of order $1$ was already observed in the previous section (see the discussion about convergence in the derivative normalisation). The existence of moments with order $<1$ and the multifractal spectrum statement are given in \cite[Corollaries 6 and 7]{DRSV14two}. The general result follows by Kahane's inequality: \cref{thm:kahane}. Note that the exponent $2dq-dq^2$ in \eqref{eqn:mf} is the limit of the corresponding subcritical exponent, which is $(d+\gamma^2/2)q-\gamma^2q^2/2$ for $\gamma<\sqrt{2d}$, \cite{RV10}. See also \cite{DRSV14two} for a proof of the \emph{KPZ-relation} satisfied by critical chaos (at least for $\star$-scale invariant fields or the planar Gaussian free field).

\begin{theorem}[Tail behaviour, \cite{MDW19}]
For any open set $A\subset D$ such that $\mathrm{Leb}(\partial A)=0$, and any non-negative continuous function $g$ on $A$
\begin{equation}
\P(\int g(x) \mu'(dx) > t) = \frac{\int_A g(v) \, dv}{\sqrt{\pi d} \, t}+o(t^{-1}) \text{ as } t\to \infty
\end{equation}
\end{theorem}

This theorem is due to Mo-Dick Wong \cite{MDW19}, who also proved a universality result for the tails of subcritical multiplicative chaos \cite{MDW19subcrit}, building on the beautiful paper \cite{RVtail} of Rhodes and Vargas. The author comments in \cite[Appendix D]{MDW19} about the technical assumption on $A\subset D$ (that is equivalent to $A$ being Jordan measurable). 

There is one case where an explicit law is known for the critical chaos measure, which is the case of the GFF on the unit circle: see \eqref{eqn:gff_cov_circle}. As mentioned previously, if $X$ is $(1/\sqrt{2})$ times this GFF, then $X$ is a one-dimensional log-correlated Gaussian field as in \eqref{eqn:kx}, and one can therefore define its chaos measure $\mu'_X$.\footnote{Note there is a different convention used in \cite{Rem20}:  $(\mu_X',\mu_X^{\gamma})$ here correspond to $(\sqrt{2}Y',Y^{\sqrt{2}\gamma})$ from \cite{Rem20}.} As part of a remarkable paper by Remy, the following is shown.  

\begin{theorem}[An explicit law, \cite{Rem20}]\label{thm:gumbel}
	$\ln(\sqrt{2}\mu'_X([0,2\pi]))$ has a standard Gumbel law. Equivalently, the density of $\sqrt{2}\mu'_X([0,2\pi])$
	is given by 
	\begin{equation} f(y)= y^{-2} \e^{-y^{-1}}\I_{y\ge 0}.\end{equation}
\end{theorem}

In fact, the main result of \cite{Rem20} is an explicit expression, the Fyodorov--Bouchaud formula \cite{FB08}, for the law of subcritical GMC masses using techniques from conformal field theory. The proof of the above theorem is based on the fact that $\mu_X'$ can be expressed as a limit of subcritical measures (see the next section). \cref{thm:gumbel} also turns out to be particularly relevant in the context of some extreme value statistics: see \cref{sec:extrema}.

Finally, let us mention what is known concerning the modulus of continuity of the critical measure. This is a topic that is not yet fully understood, but results do exist in the case of the GFF on the unit circle. This field is a little easier to analyse than in the general setting, due to the exact scaling property of the covariance kernel \eqref{eqn:gff_cov_circle}. 

For this particular field at least, the modulus of continuity is another place where a distinct contrast with the subcritical regime appears. Namely, it is known that the subcritical measures ($\gamma<\sqrt{2}$) are almost surely H\"{o}lder continuous, in the sense that there exist deterministic constants $a_1(\gamma),a_2(\gamma)$ such that with probability one $$C^{-1}|I|^{a_1}\le \mu^\gamma(I)\le C |I|^{a_2}$$
for all intervals $I$ and some random but a.s.\ finite $C$, \cite{AJKS10}.  When $\gamma=\gamma_c=\sqrt{2}$, however, this is known not to be the case. It is shown in \cite{BKNSW2} that instead, for any $a<1/2$:
\begin{equation}
\P(\exists C<\infty \text{ s.t. } \mu'(I)\le C(\log(1+|I|^{-1}))^{-a} \text{ for all intervals } I)=1.
\end{equation}
This is \emph{not} shown to be optimal in \cite{BKNSW2}, but one may expect that it is, due to the analogous result proved for multiplicative cascades in \cite{BKNSW}. On the other hand, \cite{BKNSW2} shows that for any $\alpha>1/3$, on a set of points $x$ with full $\mu'$-measure,  it holds that 
\begin{equation}
\mu'(I_n(x))\ge \exp(-c\sqrt{n(\log n + \alpha \log \log n)}) \text{ eventually},
\end{equation} where $I_n(x)$ is the dyadic interval of size $2^{-n}$ containing $x$. One consequence of this is that $\mu'$ gives full mass to a set of Hausdorff dimension 0.
\section{Limits from the subcritical regime}\label{sec:der}

The main result of this section is \cref{thm:gmcasder} below; saying that critical Gaussian multiplicative chaos measures can be constructed as ``derivatives'' of subcritical measures. This is not too hard to believe, given that the critical measures are obtained by taking derivatives of \emph{approximations} to subcritical measures, and then letting the approximations converge. The problem therefore becomes one of exchanging limits: derivative and approximation. 

 It turns out however that such an exchange of limits \emph{cannot} be justified. Namely, there is a factor of $2$ appearing on the right-hand side of \eqref{eq:gmcasder} below. Roughly speaking, the reason for this is that for $\gamma<2$ there are (almost symmetric) contributions to $\mu^\gamma$, coming from points $x$ such that $X_t(x)$ stays slightly above, and slightly below, $\gamma t$. Let us denote these by $\mu^\gamma_+$ and $\mu^\gamma_-$ respectively. On the other hand, as discussed in the previous section, contributions to $\mu'$ can \emph{only} come from points $x$ such that $X_t(x)$ stays below $\gamma_c t$. As such it is actually the derivative of $\mu^\gamma_-$ (not $\mu^\gamma$) in $\gamma$ that has the law of $\mu'$. Somewhat surprisingly, the derivative of $\mu^\gamma_+$ has the law of $\mu'$ as well; this is what results in the final factor $2$. For more detailed discussion of this point, the reader is referred to \cite{Mad16} in the setting of the branching random walk, and \cite{APS18two} in the setting of chaos for the planar GFF.

\begin{theorem}
\label{thm:gmcasder}
Suppose that $X$ is a Gaussian log-correlated field with covariance as in \eqref{eqn:kx}. Then when restricted to any $K\subset D$ compact, 
\begin{equation}\label{eq:gmcasder}
\lim_{\gamma\uparrow \gamma_c} \frac{\mu^\gamma}{\gamma_c-\gamma} \to 2 \mu'
\end{equation}
in probability (in the same sense as in \cref{thm:const_gen}).
\end{theorem}

This result was conjectured in \cite{DRSV14one}, and an analogous version proven for the branching random walk in \cite{Mad16}. However, it was not proven in the Gaussian multiplicative chaos setting until \cite{APS18one,APS18two}, in which the underlying field is assumed to be a planar Gaussian free field (or a GFF on the unit circle). The papers \cite{APS18one,APS18two} make use of the special ``local set'' approximation to the free field described in \eqref{eqn:cle_approx} to transfer the result of \cite{Mad16} to the GFF. As a consequence, \cref{thm:decomp} allows this to be carried over, see \cite[Theorem 5.5]{JSW19}, to general log-correlated fields in 2-dimensions. 

In $d$-dimensions, there seems to be no general result in the literature so far. However, the same comparison strategy can be used as long as \cref{thm:gmcasder} is shown for a good enough $d$-dimensional reference field. This is exactly what will be carried out here. Namely, the idea behind the proof of \cref{thm:gmcasder} is to show the analogous result for $\star$-scale invariant fields, and then use the decomposition result \cref{thm:decomp} to draw the general conclusion. 

\begin{prop}\label{prop:gmcasder}
	Suppose that $X$ is a $\star$-scale invariant field as in \eqref{eqn:ssi}. Then the conclusion of \cref{thm:gmcasder} holds.
\end{prop}

\begin{proof}[Proof of \cref{thm:gmcasder} given \cref{prop:gmcasder}]
By \cref{thm:decomp} it is possible to write $X=S+R$ on $K$, where $S$ is a $\star$-scale invariant field as in \cref{prop:gmcasder} and $R$ is a Gaussian field with a.s.\ locally H\"{o}lder continuous realisations. Let $\mu_S^\gamma$ for $\gamma<\sqrt{2d}=\gamma_c$ be the subcritical chaoses associated to $S$ and $\mu_S'$ be the critical chaos associated to $S$ as in \cref{thm:const_gen}. Then by \cref{prop:gmcasder} it holds that \begin{equation}\label{muS}\frac{\mu^\gamma_S}{\gamma_c-\gamma}\to 2\mu'_S \text{ as } \gamma\uparrow \gamma_c\end{equation} in probability in the weak* sense (as measures restricted to $K$). Furthermore, suppose that $\mu_X^\gamma$ are the subcritical chaoses associated to $X$, constructed using some convolution approximation $(X_\eps; \eps>0)$ to $X$ say. Then by the continuity of $R$: \begin{align}\label{muX}\mu^\gamma_X(dx)& =\e^{\gamma R(x)-(\gamma^2/2)\E(R(x)^2)}\e^{-(\gamma^2/2) \lim_{\eps\to 0}\E(X_\eps(x)^2-R_\eps(x)^2-S_\eps(x)^2)} \mu_S^\gamma(dx)=: f_\gamma(x) \mu_S^\gamma(dx)\text{ and } \nonumber \\\mu'_X(dx) &=\e^{\gamma_c R(x)-(\gamma_c^2/2)\E(R(x)^2)}\e^{-(\gamma_c^2/2) \lim_{\eps\to 0}\E(X_\eps(x)^2-R_\eps(x)^2-S_\eps(x)^2)} \mu_S'(dx)=: f_{\gamma_c}(x) \mu_S'(dx),
\end{align}
where $f_\gamma(x)\to f_{\gamma_c}(dx)$ in probability (for the topology of continuous functions on $K$) as $\gamma\uparrow \gamma_c$ (see \cite[\S 5]{JSW19}). Thus, combining \eqref{muS} and \eqref{muX}, the result follows from \cite[Lemma 5.2 (iii)]{JSW19}. This lemma is the natural statement that if $f_n\to f$ in probability (as continuous functions on a compact set $K$) and $\mu_n\to \mu$ in probability (in the weak* sense for measures on a compact set $K$) then $f_n\mu_n \to f\mu$ in probability in the weak* sense as measures on $K$.
\end{proof}

\subsection{Uniform moment bounds in $\gamma$}
This section includes some technical moment bounds that are required for the proof of \cref{prop:gmcasder}. The strategy is similar to that used in \cite{APS18two}, although the details are somewhat different.

\begin{lemma}\label{lem:p-1}
	Let $p_\gamma:=1+\frac{\gamma_c-\gamma}{\gamma_c}\in (1,2)$, and suppose that $X$ is a $\star$-scale invariant field as in \eqref{eqn:ssi}, with $\mu^\gamma$ its associated subcritical chaos measure. 
	Then for $\gamma\ge 1$ there exists a constant $C$ not depending on $\gamma$, such that
	\begin{equation}\label{eqn:moment_p}\E\big((\int_{B(0,1)}|y|^{-\gamma^2} \mu^\gamma(dy))^{p_\gamma-1}\big)\le C\end{equation}
\end{lemma}

\begin{proof}
	Recall that the covariance kernel of $X$ is of the form 
	\begin{equation}
	K(x,y)=\int_1^{\infty} \frac{k(u(x-y))}{u} \, du
	\end{equation}
	 with $k\in C^1(\R^d)$ rotationally symmetric and supported inside $B(0,1)$, such that $k(0)=1$ and $(x,y)\mapsto k(x-y)$ is a covariance on $\R^d$.
	 Let us assume that $x\cdot \nabla k(x)\le 0$ on $B(0,1)$: this is no loss of generality since the inequality must hold in $B(0,a)$ for some $a>0$, and the result in \eqref{eqn:moment_p} is clearly true if the integral is restricted to $B(0,1)\setminus B(0,a)$.
	 
	 Recall also the $\star$-scale cut-off approximations $(X_t;\, t\ge 0)$. Then \cite[Lemma 16]{DRSV14one} says that for $s\ge -\log(|y|)$, $X_{s}(y)$ can be decomposed as 
	 \begin{equation} -\int_0^{-\log|y|} g_u(y) X_u(0) \, du + Z^y + (X_s(y)-X_{-\log|y|}(y)) \end{equation}
	 where $-g_u$ is non-negative with $-\int_0^{-\log|y|} g_u=k(y)\le 1$, $Z^y$ is a centred Gaussian, independent of $(X_u(0);\, u\ge 0)$ with variance bounded by some constant $C$ independently of $y$, and $(X_s(y)-X_{-\log|y|}(y))$ is a standard linear Brownian motion independent of $Z^y$  and $(X_u(0);\, u\ge 0)$. This gives mathematical content to the heuristic that the Brownian motions $X_s(0)$ and $X_s(y)$ are ``the same'' until time $-\log|y|$, and after that, evolve independently.
	 
	 Since $p_\gamma-1\in (0,1)$, it holds by (conditional) Jensen's inequality that writing $\cG=\sigma((X_u(0); u\ge 0))$
	 \begin{align}
	 \label{align_final}
	 \E((\int_{B(0,1)}|y|^{-\gamma^2} \mu^\gamma(dy))^{p_\gamma -1} \, | \, \cG) & \le \E(\int_{B(0,1)}|y|^{-\gamma^2} \mu^\gamma(dy)\, | \, \cG)^{p_\gamma-1} \nonumber \\
	  & \lesssim ( \int_{B(0,1)} |y|^{-\gamma^2/2} \e^{-\gamma \int_0^{-\log|y|}g_u(y)X_u(0) \, du} \, dy)^{p_\gamma-1}\nonumber \\
	  & \le \e^{\gamma(p_\gamma-1) \sup_{u\in[0,-\log|y|]}X_u(0)}( \int_{B(0,1)} |y|^{-\gamma^2/2} \, dy)^{p_\gamma-1}
	 \end{align}
	 with the implied constant in the second line not depending on $\gamma$. The second line has also used the fact that $\exp(-(\gamma^2/2)\var(X_{-\log|y|}(y)))=|y|^{-\gamma^2/2}$.
	 
	 To bound the expectation of \eqref{align_final}, for each $n\ge 0$, set $r_n:=2^{-n(p_\gamma-1)^{-2}}$ and  $A_n:=\{y\in B(0,1): |y|\in (r_{n+1},r_n]\}$. Then by sub-additivity of the function $x\mapsto x^{p_\gamma-1}$, the expectation of \eqref{align_final} is less than 
	 \begin{equation}
	 \sum_n \frac{r_n^{(d-\gamma^2/2)(p_\gamma-1)}}{(d-\frac{\gamma^2}{2})^{p_\gamma-1}}\E(\e^{\gamma(p_\gamma-1) (\sup_{u\le\log(r_{n+1}^{-1})} X_u(0))}),	 \end{equation}
	 where since $(X_u(0); u\ge 0)$ is a standard Brownian motion, the expectation in the above is less than a constant times 
	 \begin{equation}
	 r_{n+1}^{\frac{\gamma^2}{2}(p_\gamma-1)^2}=r_n^{\frac{\gamma^2}{2}(p_\gamma-1)^2}2^{-\frac{\gamma^2}{2}(p_\gamma-1)^{-2}(p_\gamma-1)^2}=r_n^{\frac{\gamma^2}{2}(p_\gamma-1)^2}2^{-\frac{\gamma^2}{2}}.
	 \end{equation} Finally, observing that
	  \begin{equation}
	 (d-\gamma^2+\frac{\gamma^3}{2\sqrt{2d}})=\sqrt{2d}(p_\gamma-1)(\frac{\sqrt{2d}+\gamma}{2}-\frac{\gamma^2}{2\sqrt{2d}})\ge \sqrt{2d}(p_\gamma-1)\frac{\gamma}{2},
	 \end{equation} and that $(d-\frac{\gamma^2}{2})^{p_\gamma-1}$ is uniformly bounded in $\gamma$, it follows that the left-hand side of \eqref{eqn:moment_p} is bounded by an absolute constant (not depending on $\gamma$) times 
	 \begin{equation}
	 \sum_n r_n^{(p_\gamma-1)(d-\gamma^2/2+\frac{\gamma^3}{2\sqrt{2d}})}=\frac{1}{1-2^{-(p_\gamma-1)^{-2}(d-\gamma^2/2+\frac{\gamma^3}{2\sqrt{2d}})}}\le \frac{1}{1-2^{-\sqrt{2d}\frac{\gamma}{2}}},
	 \end{equation}
	 which is indeed uniformly bounded for $\gamma\ge 1$.
\end{proof} 

\begin{cor}\label{cor:p-1}
	Take the same set-up as \cref{lem:p-1}. Then there exists a constant $C$, independent of $\gamma\in [1,\sqrt{2d})$, such that for any non-negative $f$ on $[0,1]^d$:
	\begin{equation} \label{eqn:momfp}
	\E\big((\int_{[0,1]^d} f(x) \mu^\gamma(dx))^{p_\gamma}\big)\le C  \int_{[0,1]^d} f(x)^{p_\gamma} \, dx.	\end{equation}
\end{cor}

\begin{proof}
	First, by Jensen's inequality the left-hand side of \eqref{eqn:momfp} is less than or equal to \begin{equation}\E(\mu^\gamma([0,1]^d)^{p_\gamma} \int_{[0,1]^d} f(x)^{p_\gamma} \frac{\mu^\gamma(dx)}{\mu^\gamma([0,1]^d})
	=  \E_{\Q^{*}}(f(x^*)^{p_\gamma} \mu^\gamma([0,1]^d)^{p_\gamma-1} )
	\end{equation}
	where $\Q^*$ is as defined in \cref{sec:spine}. Recall that under $\Q^*$ and conditionally on $x^*$, the field keeps the same covariance structure as under $\P$ but has mean given by $\gamma K_X(x,x^*)$ at $x\in \R^d$. By conditioning on $x^*$ (whose marginal law is just given by Lebesgue measure on $[0,1]^d$), and by \cref{lem:p-1} together with translation invariance of the field, the result follows.
\end{proof}

\subsection{Proof of \cref{prop:gmcasder}}

The proof of this proposition follows closely the outline of \cite{Mad16,APS18two}, but in a continuum setting: making use of the decorrelation properties of the field and the $\star$-equation \eqref{eqn:star}.

\begin{proof}[Proof of \cref{prop:gmcasder}] Without loss of generality, let us show that 
\begin{equation}
 \frac{\mu^\gamma([0,1]^d)}{\sqrt{2d}-\gamma}\to 2\mu'([0,1]^d) \text{ in probability as } \gamma\uparrow \sqrt{2d}.
\end{equation}
Writing $\sqrt{t}=C/(\gamma_c-\gamma)$, the strategy is to prove that
\begin{align}\label{diag1}
\lim_{C\to \infty} \lim_{\gamma\uparrow \gamma_c} 
\, &\P(|\frac{\mu_t^\gamma([0,1]^d)}{\gamma_c-\gamma}- 2\mu'([0,1]^d|> \eps)  =0 \text{ and }\\ \label{diag2} \limsup_{C\to \infty} \limsup_{\gamma\uparrow \gamma_c} \, &\P(|\frac{\mu_t^\gamma([0,1]^d)-\mu^\gamma([0,1]^d)}{\gamma_c-\gamma}|>\eps)  =0 \text{ for any } \eps>0,
\end{align}
which clearly implies the result.

To see \eqref{diag1}, the idea is to make use of \cref{lem:sh}, writing the left-hand side of \eqref{diag1} as
\begin{equation}
\mu_t'([0,1]^d) \times \frac{e^{-C^2/2}}{C} \times \frac{\sqrt{t}}{\mu_t'([0,1]^d)} \int \e^{\gamma_c X_t-(\gamma_c^2/2)t} \e^{C (\frac{-X_t+\gamma_c t}{\sqrt{t}})}
\end{equation}
in a form reminiscent of \eqref{eqn:sh}.
Applying \cref{lem:sh} to the function $x\mapsto\exp(Cx)$, it then follows that the above converges to 
\begin{equation}\label{eqn:limwC} \mu'([0,1]^d) \times \frac{e^{-C^2/2}}{C} \times \sqrt{\frac{2}{\pi}} \E(e^{CR_1}) \end{equation}
in probability as $\gamma\uparrow \gamma_c$ (and hence $t\to \infty$). In fact, a small extra argument is required here since the function $x\mapsto\exp(Cx)$ is not bounded, but one can first truncate the function and then take a limit as the truncation lifts (exactly as in \cite{APS18two}; the details are omitted).
Since \begin{equation}\label{eqn:BMlap}\E(\e^{CR_1})\sim \sqrt{2\pi} C \e^{C^2/2},\end{equation} 
as $C\to \infty$, \eqref{eqn:limwC} converges to $2\mu'([0,1]^d)$ as $C\to \infty$, and \eqref{diag1} has been shown.

\eqref{diag2} is a little trickier, and makes use of some slightly delicate moment analysis: this is where \cref{cor:p-1} comes into play.  First observe that it is possible, for every $t\ge 0$, to cover the set $[0,1]^d$ with a finite number $N(d)$ of collections of boxes $(\cB_{t}^i)_{1\le i \le N(d)}$, where each set $\cB_t^i$ consists of $\e^{dt}$ boxes with side lengths $e^{-t}$, that are all at distance greater than $e^{-t}$ from one another. Note that $N(d)$ does not depend on $t$.

Then it is clear that 
\begin{equation} |\frac{\mu_t^\gamma([0,1]^d)-\mu^\gamma([0,1]^d)}{\gamma_c-\gamma}|\le \sum_{i=1}^{N(d)}|\frac{\mu_t^\gamma({\cB_t^i})-\mu^\gamma( {\cB_t^i})}{\gamma_c-\gamma}|.\end{equation}
Let $\cF_t$ be the $\sigma$-algebra generated by $(X_s(x); {x\in [0,1]^d, s\in [0,t]})$. It is enough to show that 
\begin{equation}\label{eq:p} \limsup_{C\to \infty}\limsup_{\gamma\uparrow \gamma_c} \P(\sum_{i=1}^{N(d)} \frac{\E(|\mu_t^\gamma(\cB_t^i)-\mu^\gamma(\cB_t^i)|^{p_\gamma} \, | \, \cF_t)}{(\gamma_c-\gamma)^{p_\gamma}}>\delta) =0 \end{equation}
for any $\delta$, where ${p_\gamma}=1+\frac{\gamma_c-\gamma}{\gamma_c}$. Indeed, for a general random variable $X$, $\sigma$-algebra $\cF$ and $p>1$
\begin{equation}
\P(|X|>\eps)\le \P(\E(|X|^p|\cF)>\eps^p \eta)) +\P(\{|X|>\eps\}\cap \{\E(|X|^p|\cF)>\eps^p \eta\}),
\end{equation}
where the second probability above is less than $\eta$ by Markov's inequality, and $\eta$ can be taken arbitrarily small. Note that ${p_\gamma}<(\gamma_c/\gamma)^2$ and so by \cite[Proposition 3.5]{RV10} (or just \cref{cor:p-1}), the ${p_\gamma}$th moments of $\mu^\gamma$ are finite. 

Next write, for fixed $i$:
\[ |\mu_t^\gamma(\cB_t^i)-\mu^\gamma(\cB_t^i)|=\left|\sum_{B\in \cB_t^i} \left(\int_{B} \e^{\gamma X_t(x)-(\gamma^2/2)t} \mu^{\gamma,t}(dx) -  \int_{B} \e^{\gamma X_t(x)-(\gamma^2/2)t}\, dx \right)\right|.\]
The key point is that, given $\cF_t$, the terms  $(\int_{B} \e^{\gamma X_t(x)-(\gamma^2/2)t} \mu^{\gamma,t}(dx) -  \int_{B} \e^{\gamma X_t(x)-(\gamma^2/2)t}\, dx )$ for each $B\in \cB_t^i$ are conditionally independent, and have conditional mean $0$. Since $p_\gamma\in (1,2)$ this allows for an application of the (conditional) von Bahr--Esseen theorem \cite{EVB} to obtain that
\begin{align} \label{align1}& \frac{\E(\big|\mu_t^\gamma(\cB_t^i)-\mu^\gamma(\cB_t^i)|^{p_\gamma}\, | \, \cF_t)}{(\gamma_c-\gamma)^{p_\gamma}}  \nonumber \\ & \le (\frac{2}{\gamma_c-\gamma})^{p_\gamma} \sum_{B\in \cB_t^i} \E(|\int_{B} \e^{\gamma X_t(x)-(\gamma^2/2)t}\e^{-dt} \mu^{\gamma,t}(dx) -  \int_{B} \e^{\gamma X_t(x)-(\gamma^2/2)t}\, dx \big|^{p_\gamma} \, | \, \cF_t) \nonumber \\ & \le (\frac{2}{\gamma_c-\gamma})^{p_\gamma} \sum_{B\in \cB_t^i} \left(\E(|\int_{B} \e^{\gamma X_t(x)-(\gamma^2/2)t} \e^{-dt} \mu^{\gamma,t}(dx)|^{p_\gamma} \, | \, \cF_t)+ |\int_{B} \e^{\gamma X_t(x)-(\gamma^2/2)t}\, dx |^{p_\gamma}\right). 
\end{align}

Now by Jensen's inequality, it is possible to bound 
\[ |\int_{B} \e^{\gamma X_t(x)-(\gamma^2/2)t}\, dx |^{p_\gamma}\le |B|^{{p_\gamma}-1}\int_B \e^{\gamma {p_\gamma}X_t(x)-\frac{\gamma^2{p_\gamma}}{2}t} dx\le \int_B \e^{\gamma {p_\gamma}X_t(x)-(\frac{\gamma^2{p_\gamma}}{2}+d({p_\gamma}-1))t} dx ,\]
and exactly the same bound holds for $\E(|\int_{B} \e^{\gamma X_t(x)-(\gamma^2/2)t} \e^{-dt} \mu^{\gamma,t}(dx)|^{p_\gamma} \, | \cF_t)$ by \cref{cor:p-1} and the scaling property \eqref{eqn:ssi} of $\mu^{\gamma,t}$. So, putting this together, it follows that the left-hand side of \eqref{align1} is less than or equal to 
\begin{align} &(\frac{2}{\gamma_c-\gamma})^{p_\gamma} \sum_{B\in \cB_t^i} \int_B \e^{\gamma {p_\gamma} X_t(x)-\frac{\gamma^2{p_\gamma}}{2}+d({p_\gamma}-1))t}\, dx  \nonumber \\ &= 2 \,(\frac{2}{\gamma_c-\gamma})^{{p_\gamma}-1}\,\frac{\gamma_c-\gamma {p_\gamma}}{\gamma_c-\gamma}  \, \frac{\e^{(\frac{\gamma^2{p_\gamma}^2}{2}-\frac{\gamma^2{p_\gamma}}{2}-d({p_\gamma}-1))t}}{\gamma_c-\gamma {p_\gamma}}\,\int_{[0,1]^d} \e^{(\gamma {p_\gamma})X_t(x)-(\gamma^2{p_\gamma}^2/2)t} \, dx. \end{align}

Finally, observe that
\begin{equation}\label{eqn:expC2} \e^{(\frac{\gamma^2{p_\gamma}^2}{2}-\frac{\gamma^2{p_\gamma}}{2}-d({p_\gamma}-1))t}=\e^{-\frac{C^2}{2}(\frac{\gamma_c+\gamma}{2}-\frac{\gamma^2}{\gamma_c^2})}\to \e^{-\frac{C^2}{2}} \text{ as } t\to \infty, \end{equation}  while $(2/(\gamma_c-\gamma))^{{p_\gamma}-1}\to 1$ and $(\gamma_c-\gamma {p_\gamma})/(\gamma_c-\gamma)\to 1$ as $t\to\infty$. Since $\gamma {p_\gamma} \uparrow \gamma_c$ as $\gamma\uparrow \gamma_c$ it also holds by \eqref{eqn:limwC} that 
\begin{equation}  \frac{1}{\gamma_c-\gamma {p_\gamma}} \int_{[0,1]^d} \e^{\gamma {p_\gamma} X_t(x)-\frac{\gamma^2{p_\gamma}^2}{2}t} \, dx \to \mu'([0,1])^d \e^{-\tilde{C}^2/2}{\tilde C}\sqrt{\frac{2}{\pi}} \E(\e^{\tilde{C}R_1}) \end{equation}
in probability as $\gamma\uparrow \gamma_c$, where $\tilde{C}$ is such that $C/(\gamma_c-\gamma)=\tilde{C}/(\gamma_c-\gamma p_\gamma)$.
Since the right-hand side of the above tends to $2\mu'([0,1]^d)<\infty$ as  $C$ and (therefore $\tilde{C}$) tends to infinity,  combining this with \eqref{eqn:expC2} and a union bound provides \eqref{eq:p}. 
\end{proof}

\section{Applications} 

\subsection{Studying extrema}\label{sec:extrema}
One of the areas in which critical chaos measures turn out to play a distinguished role, is in the study of extreme values of log-correlated fields. It is believed that the behaviour of these extrema should be somewhat universal, not only within the world of \emph{Gaussian} log-correlated fields, but extending to log-correlated models in random matrix theory and even probabilistic number theory.
This section is intended to give a flavour of what is known and expected to be true, but is in no way comprehensive, and the author has little claim to expertise in this area so the discussion will be kept at a high level. For a more thorough exposition of this topic the reader is referred to \cite{Arg16}; see also  \cite[\S 2]{BL16}.

\subsubsection{Maxima of Gaussian log-correlated fields}

It has been known since the work of Bramson \cite{bra83} that the position of the minimal (or maximal) particle in a branching Brownian motion has a limiting speed $\sqrt{2}$, that its median has a (negative) logarithmic second order correction, and that the difference between the minimal position and its median converges in distribution to a random shift of a Gumbel distribution \cite{LS87}. The same result is known for the minimal position in a general branching random walk, due to A\"{i}d\'{e}kon and building on earlier work of many authors; see \cite{Aid13} and the references therein. Moreover, the random shift of the limiting Gumbel distribution is given by the limit of the ``derivative martingale'' for the branching random walk. That is, the branching random walk analogue of the critical chaos measure.

As ever, a similar phenomenon is expected to be seen for log-correlated Gaussian fields. In fact there has already been cause to consider extreme values of the field in this article, recall \eqref{eqn:extrema_basic}, but this was a only a preliminary estimate.\\

Let us start by stating a general tightness result.

\begin{theorem}[\cite{Aco14}]\label{thm:tight_max}
	Suppose that $(Y_\eps;\, \eps>0)$ are centred Gaussian fields on $[0,1]^d$ for each $\eps$, such that for some $C_Y<\infty$: $\cov(Y_\eps(x)Y_\eps(y))-\log(\eps \vee \|x-y\|)\le C_Y$ for all $x,y,\eps$; and $\E[(Y_\eps(x)-Y_\eps(y))^2]\le C_Y\eps^{-1}\|x-y\|$ for all $\|x-y\|<\eps$. Then, setting
	\begin{equation}
	m_\eps:=\sqrt{2d}\log(1/\eps)-\frac{3/2}{\sqrt{2d}}\log\log(1/\eps),
	\end{equation}
	the family $(\sup_{x\in [0,1]^d} Y_\eps(x)) - m_\eps$ is tight in $\eps$. 
\end{theorem}
Given this, the natural question to ask is whether the limit of $\sup Y_\eps - m_\eps$ exists in law as $\eps\to \infty$. Motivated by the above discussion one may also wonder if there is a universal feature, related to a Gumbel distribution and critical Gaussian multiplicative chaos, in the limit.

To the best of the authors knowledge this question is yet to be completely resolved. But there has been considerable progress:

\begin{itemize}
	\item In \cite{Mad15}, Madaule proves that for $\star$-scale cut-off approximations to $\star$-scale invariant log-correlated Gaussian fields,
	\begin{equation}\label{eqn:gumbel}
	\sup_{x\in [0,1]^d}(X_t(x)-\sqrt{2d}t+\frac{3}{2\sqrt{2d}}\log(t)) \to G_d
	\end{equation}
	in law as $t\to \infty$, where $G_d$ is a Gumbel distribution convolved with $\mu'([0,1]^d)$ and $\mu'$ is the critical measure constructed in \cref{thm:const_gen}. More precisely, there exists a constant $C^*$, such that for all $z\in \R$: 
	\begin{equation}\label{eqn:starscalemax}
	\lim_{t\to \infty} \P(	\sup_{x\in [0,1]^d}(X_t(x)-\sqrt{2d}t+\frac{3}{2\sqrt{2d}}\log(t))\le -z)=\E(\exp(-C^*\e^{\sqrt{2d}z}\mu'([0,1]^d))).
	\end{equation}
	\item The analogous result has been shown by Bramson, Ding and Zeitouni \cite{BDZ16} and Biskup and Louidor \cite{BL16} for  discrete approximations (i.e., on a lattice) to the planar Gaussian free field. Again the critical measure for the free field appears as the random shift of the limiting Gumbel distribution \cite{BL20}. In fact, the papers \cite{BL16,BL20} describe not only the maximum of the field, but the extremal process associated to its local maxima, unveiling an even stronger connection with critical GMC. 
	
	To explain this, take a domain $D\subset \C$, and for every $N$ a discrete Gaussian free field $h_N$ defined on a lattice version $D_N$ of $N D$ (see \cite{BL16} for the detailed definition of $D_N$). More precisely, $h_N$ is the centred Gaussian field on $D_N$ with probability density at $(h_N(v))_v$ proportional to \begin{equation}
	\label{def:discGFF}
	\exp(-\frac{1}{2} \sum_{v\in D_N} \nabla h_N(v))
	\end{equation}for each $N$, where $\nabla$ is the discrete gradient on $\Z^2$ and $h_N(v)$ is set to $0$ outside of $D_N$ by convention.
	
	 Then Biskup and Louidor prove that the rescaled positions and recentred values of the local maxima of $h_N$ converge, as $N\to \infty$, to a Poisson point process with intensity given by 
	\begin{equation}
	Z^D(dx)\otimes e^{-\sqrt{2\pi}h} \, dh,
	\end{equation}
	where $Z^D$ is the critical chaos measure for a (continuum) GFF on $D$. The identification of $Z^D$ turns out to be highly non-trivial matter, and was actually completed in \cite{BL20} using a characterisation theorem.
	
	From this it follows, in particular, that for any $A\subset D$ and $z\in \R$:
	\begin{equation}
	\P(N^{-1} \arg\max_{D_N} h_N\in A; \, \max_{x\in D_N} h_N(x)- m_N \le -z )\to \E(\frac{Z^D(A)}{Z^D(D)}\exp(-\frac{1}{\sqrt{2\pi}} \e^{\sqrt{2\pi}z} Z^D(D))
	\end{equation}
	as $N\to \infty$, where $m_N:=\sqrt{8/\pi}\log N - \sqrt{3/(2\pi)}\log \log N$; cf. \eqref{eqn:starscalemax}. Let us also remark that in \cite{BL18}, the authors prove a further convergence result for the \emph{full} extremal process of the discrete GFF (i.e., not restricted to local maxima) which allows them to prove a local limit theorem for the position and value of the GFF maximum.

	\item Ding, Roy and Zeitouni \cite{DRZ17} were also able to build on the techniques of \cite{BDZ16}, and extend the result for the planar GFF to much more general sequences of log-correlated Gaussian fields on large lattices (equivalently, fine lattice approximations to a bounded set). The reader is referred to \cite{DRZ17} for the exact assumptions made, but let us note that in some natural circumstances they require significant work to verify. 
	
	\item For example, this is the case when considering the four-dimensional membrane model \cite{Sch20}: a Gaussian model for a random height function defined on a subset $A$ of $\Z^4$. The exact definition is closely related to that of the planar discrete GFF, having probability density at $(h_v)_{v\in A}$  proportional to $\exp(-\frac{1}{2}\sum_{v\in A}|\Delta h_v|^2)$ rather than \eqref{def:discGFF} (see \cite{Sch20} for details). In \cite{Sch20}, it is proven that the recentred maximum of this function converges to a randomly shifted Gumbel distribution, by showing that the conditions of \cite{DRZ17} are satisfied. This requires a substantial and delicate analysis of the Green's function for the bi-Laplacian on $\Z^4$.
	\item As far as the author understands, a completely general result concerning the convergence of recentred maxima, say for fields as in \cref{thm:tight_max}, is not currently known.
\end{itemize}

As hinted at previously, understanding the law of critical Gaussian multiplicative chaos measures may not only help to understand the extreme values of log-correlated Gaussian fields. Namely, there are other non-Gaussian fields that are expected to be governed by similar behaviour. 

In the coming subsections, a few different families of such fields will be discussed. Let us emphasise again that each of these topics constitutes a rich and substantial area of research in its own right. Thus, what follows represents just a snapshot of progress, and will focus only on the results that are most closely linked with critical Gaussian multiplicative chaos.

\subsubsection{Ginzburg--Landau model}
As an initial example, one family of approximately log-correlated fields (that are actually known to converge to continuum Gaussian free fields under certain conditions \cite{Mil11,NS97}) are gradient models with convex interactions or ``Ginzburg-Landau'' fields.  These are natural generalisations of the discrete planar GFF, where the density $\exp(-\frac{1}{2}\sum|\nabla h|^2)$ for the free field is replaced by \begin{equation} \exp(-\frac{1}{2}V(\nabla h)) \; ; \; V(\mathbf{x})=v(x_1)+v(x_2)\end{equation} for some $v:\R\to \R$ symmetric and convex. 

Given the close connection with the planar Gaussian free field, it is natural to ask if the maxima of these Ginzburg--Landau fields display similar characteristics. Substantial progress on this was recently made in \cite{WW19}, building on the previous work \cite{BW16}. Under the assumption that $v\in C^2$ has bounded elliptic contrast ($v''\in [c,C]$ for some $0<c<C<\infty$), \cite{WW19} proves that after recentring around its mean, the maximum of the field in $[-N,N]^2\cap \Z^2$ is tight along a deterministic subsequence. In light of the discussion in the previous subsection, it is reasonable to believe that the  recentred maximum should actually converge, and that its limit will be a shifted Gumbel distribution. However, this seems to remain open.
 
\subsubsection{Random matrices}
 
Let us turn next to random matrix theory, a topic whose connection to Gaussian fields and GMC has been widely studied and utilised in recent years. The relationship between the two comes about when looking at the characteristic polynomials of large random matrices sampled from natural ensembles. More concretely, it is expected and in some settings proved - see for example \cite{HKO01,RV07,FKS16} - that the logarithms of these characteristic polynomials will behave like Gaussian log-correlated fields as the matrix size goes to $\infty$. 

Building on this idea, it should therefore be the case that normalised powers of such characteristic polynomials converge to Gaussian multiplicative chaos measures. This has now been shown rigorously in several settings when the power corresponds to a \emph{subcritical} chaos measure \cite{Web15,NSW18,LOS18,BWW18,CFLW19,Lam19,Lam20}. Recently in the case of the circular $\beta$-ensemble, convergence (of a closely related quantity) to a \emph{critical} chaos measure has been shown, \cite{CN19}. The proof of this actually goes via a subcritical version, making use of \cref{thm:gmcasder} for the GFF on the circle. Similar results are naturally expected to hold for a large class of matrix ensembles, as in the subcritical setting, but this seems to be the only result to date. 

More generally, it is  believed that the extreme values of log characteristic polynomials for  large random matrices should display the same behaviour as extrema of Gaussian log-correlated fields. For example, Fyodorov, Hiary and Keating famously conjectured \cite{FHK12} that if $X_N$ is the characteristic polynomial of a Haar-distributed $N\times N$ random unitary matrix, then
\begin{equation}\label{conj:fhk} \sup_{|z|=1} \log|X_N(z)| - (\log N - \frac{3}{4} \log \log N)
\end{equation}
will converge to a limiting law as $N\to \infty$, which is the average of two independent Gumbel variables. This should be compared with \eqref{eqn:gumbel} and \cref{thm:gumbel}: recall that the latter describes the law of the mass of a critical GMC measure in terms of a Gumbel distribution.
This conjecture has not been proven to date, but the sequence has been shown to be tight and significant progress on bounds has been achieved, \cite{PZ18,CMN18}. Let us also mention the works \cite{CFLW19,Lam20} that use GMC techniques to identify the leading order of the maximum for the characteristic polynomials of Hermitian matrices and the Ginibre ensemble respectively.

\subsubsection{Riemann-zeta function}

In a somewhat different direction, it turns out that the Riemann-zeta function $\zeta$, recentred at a random point on the critical line, can be connected with Gaussian multiplicative chaos. More precisely, one can consider a uniform random variable $\omega\in [1,2]$, and study what 
\begin{equation}\label{eqn:zeta}\log \zeta(1/2 + i\omega T + ix)\end{equation} looks like as $x$ ranges over an interval of some size and $T\to \infty$. The broad principle, see \cite{FHK12,FK14} for some heuristics, is that (at least the real part of) this can be compared with a Gaussian log-correlated field in the limit. One motivation for such a study is the major open problem in number theory of understanding the growth and \emph{global} maxima of $\zeta$ on the critical line. This will not be touched upon here, but interested readers are referred to \cite[\S 1.1]{AOR19} for a summary, and the references therein for more details.

As far as mathematical results concerning \eqref{eqn:zeta} go, Selberg's classical theorem \cite{Sel92} guarantees that $\log|\zeta(1/2+i\omega T)|/\sqrt{\log \log T}$ converges to a centred Gaussian random variable as $T\to \infty$. Since then, many precise correlation results have been proven. For example:
\begin{itemize}
	\item \cite{HNY08} showed the independence of values of $\log|\zeta(1/2+iy)|$ at points $y$ that are sufficiently far apart;
\item 	\cite{Bou10} showed joint convergence to a correlated Gaussian vector, of the values at points that are sufficiently close;
\item \cite{AOR19} showed the existence of a freezing transition for the moments, over both macroscopic and mesoscopic intervals, as predicted by \cite{FK14}. More precisely, they showed that for $\theta>-1$ and any $\eps>0$, 
\begin{equation}\label{eqn:ouimet}
\int_{-(\log T)^{\theta}}^{(\log T)^{\theta}} |\zeta(\frac{1}{2}+i\omega T +ix)|^\beta \in [(\log T)^{f_\theta(\beta)-\eps},(\log T)^{f_\theta(\beta)+\eps}] \text{ with probability } 1-o(1)
\end{equation}
as $T\to \infty$,  where $f_\theta(\beta)$ changes from being quadratic to linear at $\beta=2$ when $\theta\le 0$, and at $\beta=2\sqrt{1+\theta}$ for $\theta\in (0,3]$.  See \cite[Theorem 1.1]{AOR19} for the explicit description of $f_\theta(\beta)$. Note that when $\theta=0$ (considering an interval of fixed size), if $\Re \log \zeta(1/2+i\omega T+ix)$ is comparable to a log correlated Gaussian field, then this freezing transition corresponds to the phase transition for Gaussian multiplicative chaos (\cref{sec:phase_trans}).

\end{itemize} 

From a slightly different perspective, Saksman and Webb \cite{SW16i,SW16ii} considered the \emph{truncation } $\zeta_N(s)=\prod_{k=1}^N (1-p_k^{-s})^{-1}$  of the Riemann-zeta function, and were able to show (among other things) that $\zeta_N(1/2+i\omega T+i x)$ converges to a limiting generalised function $\zeta_{N,\mathrm{rand}}(1/2+ix)$ as $T\to \infty$. They concretely connected this to Gaussian multiplicative chaos by proving that the random measures 
\begin{equation}
\frac{|\zeta_{N,\mathrm{rand}}(1/2+ix)|^\beta}{
\E(|\zeta_{N,\mathrm{rand}}(1/2+ix)|^\beta)} \, dx  \text{ and } \sqrt{\log\log N}\frac{|\zeta_{N,\mathrm{rand}}(1/2+ix)|^2}{
\E(|\zeta_{N,\mathrm{rand}}(1/2+ix)|^2)} \, dx
\end{equation} 
converge to subcritical and critical  GMC measures (respectively) on $[0,1]$  as $N\to \infty$. On the other hand, while they conjecture that in fact
\begin{equation}
(\log T)^{-\frac{1}{4} \beta^2}\, |\zeta(1/2+ix+i\omega T)|^\beta dx \text{ for }\beta<2, \text{ and } \sqrt{\log \log T} |\zeta(1/2+ix+i\omega T)|^2
\end{equation}
converge to the same measures as $T\to \infty$, cf. \eqref{eqn:ouimet}, this still remains open.

Finally, there is the question of the extreme values of \eqref{eqn:zeta}. An analogous conjecture of Fyodorov, Hiary and Keating in this case (cf. \eqref{conj:fhk}), is that
 \begin{equation}\label{conj:fhk2}
 \sup_{x\in [-1,1] } \log|\zeta(1/2+i\omega T + ix)|-(\log \log T - \frac{3}{4} \log \log \log T) 
 \end{equation}
 should converge to a limiting distribution as $T\to \infty$. This also remains open, but recent progress has been made by \cite{ABH17,Naj18,ABB19,Har19}: there now exist both lower and upper bounds with high probability (to first and second order respectively) for the supremum in \eqref{conj:fhk2}. This is further extended in \cite{AOR19} to intervals $x\in[-(\log T)^\theta, (\log T)^\theta]$ of different lengths ($\theta>-1$). The reader is referred to the cited articles for precise statements. 
 
 \subsubsection{Local time of planar random walks and Brownian motion}
 
 Other models with related extreme value statistics arise by considering local times of planar random walks and Brownian motion. Due to the isomorphism theorems of Dynkin and Ray--Knight \cite{Dynkin,EKMRS}, which relate these local times to squares of discrete Gaussian free fields, it is reasonable to expect a relationship between the ``thick points'' of these objects. To summarise:
 \begin{itemize}
 \item Thick points of a planar random walk are those places that the walk visits unusually often, and are thus encoded by atypically large local times. The sets of such points (with thickness parametrised appropriately) have been the subject of extensive investigation over the years. In particular,  the sizes of these sets have now been described precisely in the works \cite{DPRZ,Ros05,BR07,JegTP}.
 
 \item As discussed in \cref{sec:const} of this article, the thick points of log-correlated Gaussian fields (see e.g., \eqref{eqn:thick}) are intimately linked with the associated Gaussian multiplicative chaos measures. More precisely, the measure with parameter $\gamma$ will give full mass to the set of $\gamma$-thick points, including at $\gamma=\gamma_c$: the largest value of $\gamma$ for which thick points actually exist.

\item By analogy with the GMC case, it is therefore natural to try and construct measures that are supported on sets of thick points for planar \emph{Brownian motion}, with good reason to believe that these should describe the distribution of random walk thick points in an appropriate scaling limit. Such a measure was first constructed via a regularisation procedure in \cite{BBK94}, for any ``thickness parameter'' $a$ less than $1/2$. Later, simultaneously in \cite{AHS18,JegCon}, the construction was extended to the full ``subcritical'' range $a\in (0,2)$. These articles are concerned with the accumulated local time for Brownian motion run until it leaves a given planar domain $D$, and although the regularisations of local time differ slightly between them, they eventually produce the same measures. These measures are sometimes referred to as ``Brownian chaos measures''. 
\item Following the construction, it was shown by Jego \cite{JegChar} that point processes of planar random walk thick points do indeed converge weakly to these measures after suitable rescaling.
\end{itemize}

 In this framework, $a=2$ corresponds to the \emph{critical} parameter, at or above which the ``usual'' approximation procedure yields a trivial measure in the limit. Although the case $a=2$ was not considered initially, critical Brownian multiplicative chaos has now been defined by Jego, \cite{JegCrit}. It is constructed in \cite{JegCrit} using both the Seneta--Heyde and ``derivative'' normalisation schemes (cf. \cref{thm:const_gen}) as well as via a limit from the subcritical regime (cf. \cref{thm:gmcasder}). To avoid confusion with the Gaussian case, let us denote this measure by $\nu'$.
 
 Of course from here, one would like to know if the ``thickest'' points of planar random walks display the same behaviour as the extrema of Gaussian log-correlated fields. With the critical measure $\nu'$ defined, Jego was able to formulate a specific conjecture for this in \cite{JegCrit}. Namely, if $l_x^N$ is the total local time at $x$, for a simple random walk on $\Z^2$ started from the origin and stopped when leaving $[-N,N]^2\cap \Z^2$, then the conjecture  is that for all $z\in \R$
 \begin{equation}\label{conj:brownian_chaos}
 \P((\sup_{x\in \Z^2} \sqrt{l_x^N} - (\frac{2}{\sqrt{\pi}}\log N - \frac{1}{\sqrt{\pi}} \log \log N)
 \le -z )\to \E(\exp(-c_1 \nu'([-1,1]^2)\e^{c_2 z})
 \end{equation}
 as $N\to \infty$, for some positive constants $c_1, c_2$. Again this would say that the limiting law of the recentred maximum is a Gumbel law shifted by a critical chaos measure.
 
Let us conclude this subsection by mentioning that there is already a growing body of work on very closely related questions. For example, \cite{abe2018} proved the analogous result to \eqref{conj:brownian_chaos} in the setting of simple random walks on symmetric trees, while \cite{CLS18,DRZ19} have recently shown the same for cover times of random walks on binary trees. In a slightly different direction, \cite{abe2015,AB19,ABL19} considered local times of random walks when they are run up to a time proportional to the the cover time of the graph. In \cite{abe2015} this is the 2d-torus, while in \cite{AB19,ABL19} it is a lattice approximation to a planar domain with all boundary vertices identified and re-entry through the boundary allowed (see \cite{AB19,ABL19} for details). In these works, among other things, thick points for the local time are shown to be distributed according to subcritical GFF chaos measures. In addition, the local structure of the local time field near these thick points is identified.

\subsection{Conformal welding}
One final topic where a surprisingly nice picture in connection with Gaussian multiplicative chaos emerges, is that of random conformal weldings. The story is quite involved and different in flavour to the rest of this survey, but nonetheless deserves a mention.

The classical (deterministic) conformal welding problem is the question of embedding a pair of discs, glued along their boundaries according to some homeomorphism $\varphi$, into the two-dimensional sphere $\mathbb{S}^2$. When this is possible, it produces a \emph{welding curve} $\eta\in \mathbb{S}^2$, which separates the images of the two discs under the embedding. There is a rich theory of complex analysis surrounding this question, addressing when such weldings exist, and studying the interplay between curve and homeomorphism. For an introduction to the topic that is particularly well-suited to probabilists, see \cite{AJKS10}.

One natural choice of homeomorphism, the case of isometric welding, is given by identifying two measures on the disc boundaries according to length with respect to some fixed reference points. Of course, one can also consider the problem with \emph{random} boundary measures, and this is where Gaussian multiplicative chaos comes in. 

When the boundary measures are irregular (random or not, but for example when they are GMC measures) it is really quite hard to determine whether or not the conformal welding problem admits a solution. And even harder to say anything about the welding curve, \cite{Bis07}.
However, there is a remarkable result due to Scott Sheffield \cite{Sh16}, which describes exactly what happens when the boundary measures are taken to be two independent copies of subcritical chaos for (a variant of) the GFF on the unit circle. A \emph{constructive} approach in this specific set-up yields both existence of the conformal welding, and classification of the welding curve as a Schramm--Loewner evolution (SLE$_\kappa$) with parameter $\kappa=\gamma^2<4$. 

On the other hand, the theorem does not extend to the case of critical chaos, which happens to correspond to a special ``transition point'' $\kappa=4$ for the  SLE parameter. This is the point below which SLE$_\kappa$ is a simple curve, and above which SLE$_\kappa$ has double points at all scales. It was only as a result of \cref{thm:gmcasder} that \cite{HP18} were able to extend Sheffield's result to the critical case. In turn, this brings about another potential application of critical chaos: can it be used to say anything about the behaviour of SLE$_4$? The regularity of SLE$_4$ actually remains somewhat mysterious, meaning that such an application would certainly be interesting (although far from straightforward).

To the best knowledge of the author, there are not any results concerning existence of conformal weldings for more general chaos measures on the circle. However, the beautiful paper \cite{AJKS10} did consider a problem closely related to that of \cite{Sh16} from a much more analytic perspective, and it may well be that their strategies can be adapted to work in a broader context.

\bibliographystyle{alpha}
\bibliography{EP_bibliography}
\end{document}